\documentclass[10pt]{article}

\usepackage{graphicx}
\graphicspath{ {./images/} }
\usepackage{lipsum}% Just for this example
\usepackage{subcaption}

\usepackage{xcolor}

\usepackage[utf8]{inputenc}

\usepackage{booktabs}

\usepackage{amsthm}
\usepackage{bbm}
\usepackage{mathrsfs}
\usepackage{dynkin-diagrams}

\usepackage[utf8]{inputenc}

\usepackage{tikz-cd}

\usepackage{enumitem}

\usepackage{amssymb}
\usepackage{float}
\usepackage{amsbsy}
\usepackage[all]{xy}\usepackage{xypic}
\usepackage{braket}
\usepackage[colorlinks = true,citecolor=black]{hyperref}
\usepackage{amsmath}

\usepackage{wrapfig}

\usepackage[upint]{stix}

\usepackage{authblk}

\DeclareMathOperator{\arccosh}{arccosh}

\hypersetup{
     colorlinks=true,
     linkcolor=blue,
     filecolor=blue,
     citecolor = blue,      
     urlcolor=cyan,
     }

\makeatletter
\newcommand\opteq[1]{\mathrel{\mathpalette\opt@eq{#1}}}
\newcommand{\opt@eq}[2]{%
  \begingroup
  \sbox\z@{$#1#2$}%
  \sbox\tw@{\resizebox{!}{.5\ht\z@}{$\m@th#1($}}%
  \nonscript\hskip-\wd\tw@
  \mkern1mu
  \raisebox{-.35\ht\z@}[0pt][0pt]{\resizebox{!}{.5\ht\z@}{$\m@th#1($}}%
  \mkern-1mu
  {#2}%
  \mkern-1mu
  \raisebox{-.35\ht\z@}[0pt][0pt]{\resizebox{!}{.5\ht\z@}{$\m@th#1)$}}%
  \mkern1mu
  \nonscript\hskip-\wd\tw@
  \endgroup
}
\makeatother

\def\XXint#1#2#3{{\setbox0=\hbox{$#1{#2#3}{\int}$}
     \vcenter{\hbox{$#2#3$}}\kern-.5\wd0}}

\newtheorem{theorem}{Theorem}[section]
\newtheorem*{theorem*}{Theorem}
\newtheorem{theorem-non}{Theorem}
\newtheorem{lemma-non}{Lemma}

\theoremstyle{definition} 
\newtheorem{thm}{Theorem}

\theoremstyle{definition} 
\newtheorem{corollarynon}{Corollary}

\newtheorem{conjecture-non}{Conjecture}

\newtheorem{corollary-non}{Corollary}
\newtheorem{proposition}[theorem]{Proposition}
\newtheorem{lemma}[theorem]{Lemma}
\newtheorem*{lemma*}{Lemma}
\newtheorem{corollary}[theorem]{Corollary}

\newtheorem*{conjecture*}{Conjecture}

\theoremstyle{definition}
\newtheorem{definition}[theorem]{Definition}
\newtheorem{example}[theorem]{Example}

\theoremstyle{remark}
\newtheorem{remark}[theorem]{Remark}

\numberwithin{equation}{section}

\usepackage{geometry}
\begin{document}
\title{{\bf{Collapsing Flat ${\rm{SU}}(2)$-Bundles to Spherical $3$-Manifolds}}}

\author{Eder M. Correa }
\affil[1]{Universidade Estadual de Campinas, Brazil\\

Instituto de Matemática, Estatística e Computação Científica\\

{\rm{ederc@unicamp.br}}}

%\thanks{ederc@unicamp.br}

\maketitle

\begin{abstract}
We present a geometric mechanism for the emergence of spherical $3$-manifolds from the superspace of Riemannian metrics associated with flat ${\rm{SU}}(2)$-bundles over closed orientable hyperbolic surfaces. Our main result shows that any homogeneous spherical 3-manifold $(S,g_{S})$ can be realized as a boundary point in the Gromov-Hausdorff closure of a superspace $\mathcal{S}(P)$, where $P$ is a flat ${\rm{SU}}(2)$-bundle over a closed orientable hyperbolic surface $(\Sigma,h_{\Sigma})$. We show that the convergence of the sequence of metric spaces towards the spherical limit is controlled by the order of the fundamental group of $S$ and the metric invariant of the hyperbolic base provided by the ratio between its area and its systole. In this framework, the problem of obtaining the sharpest upper bound error reduces to the classical problem of maximizing the systole function over the moduli space of hyperbolic Riemann surfaces. As a byproduct, we observe that certain arithmetic surfaces provide the best possible error estimates within this family. To illustrate these results, we show that, according to our mechanism, the Bolza surface yields the optimal error bound for the convergence toward the Poincaré homology sphere.
\end{abstract}

\hypersetup{linkcolor=black}

\tableofcontents

\hypersetup{linkcolor=black}

\section{Introduction}

The ADE diagrams first appeared in Coxeter's classification of reflection-generated finite groups in the 1930s, e.g. \cite{coxeter1934discrete}. They reappeared a decade later in Dynkin's classification of semisimple Lie algebras, see \cite{dynkin1947structure}. Since the 1970s, these same diagrams have continued to surface across mathematics and physics, from representation theory and algebraic geometry to quiver theory and gauge theory. For a comprehensive account of this ubiquitous and intriguing pattern, we refer the reader to \cite{cameron2025ade}.

\begin{figure}[H]
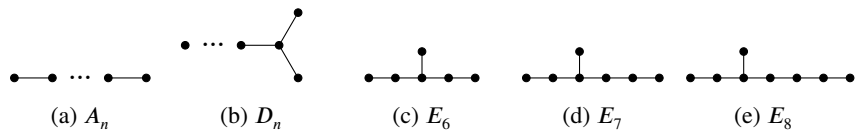

    \centering
    % Diagrama A_n (exemplo com n=5)
    \begin{subfigure}[b]{0.15\textwidth}
        \centering
        \dynkin[edgelength=0.5cm]{A}{2} $\cdots$ \dynkin[edgelength=0.5cm]{A}{2}
        \caption{$A_n$}
    \end{subfigure}
    % Diagrama D_n (exemplo com n=5)
    \begin{subfigure}[b]{0.15\textwidth}
        \centering
        \dynkin[edgelength=0.5cm]{A}{1} $\cdots$ \dynkin[edgelength=0.5cm]{D}{4}
        \caption{$D_n$}
    \end{subfigure}
    % Diagrama E6
    \begin{subfigure}[b]{0.15\textwidth}
        \centering
        \dynkin{E}{6}
        \caption{$E_6$}
    \end{subfigure}
    % Diagrama E7
    \begin{subfigure}[b]{0.15\textwidth}
        \centering
        \dynkin{E}{7}
        \caption{$E_7$}
    \end{subfigure}
    % Diagrama E8
    \begin{subfigure}[b]{0.15\textwidth}
        \centering
        \dynkin{E}{8}
        \caption{$E_8$}
    \end{subfigure}
    \caption{ADE type Dynkin diagrams.}
\end{figure}

A particularly notable instance is the McKay correspondence \cite{mckay1980graphs}, which establishes a direct link between the finite subgroups of ${\rm{SU}}(2)$, the binary polyhedral groups, and the ADE diagrams. These groups, namely the cyclic, binary dihedral, binary tetrahedral, binary octahedral, and binary icosahedral groups, are the building blocks for some of the most symmetric geometric objects in three dimensions.
\begin{center}
\begin{figure}[H]
\centering\includegraphics[scale = .30]{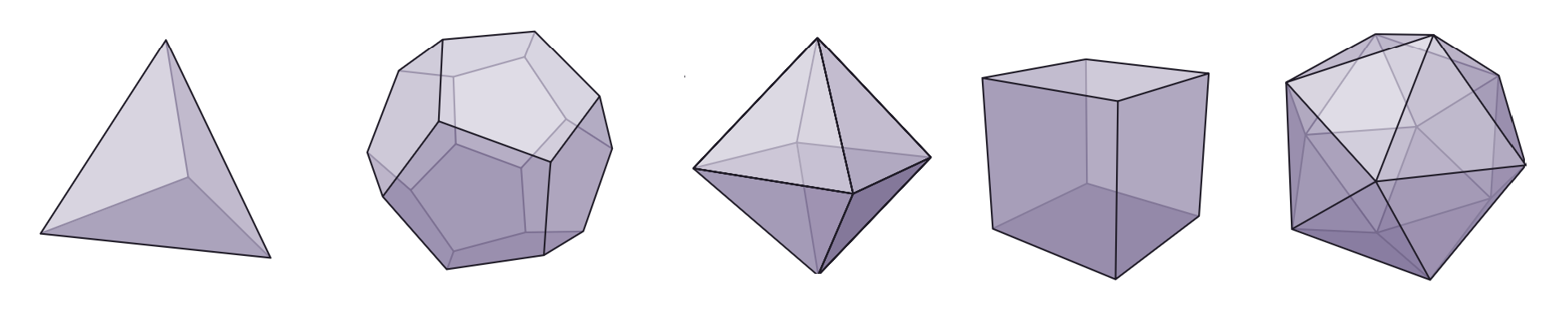}
\caption{The symmetry groups of the Platonic solids correspond, via the McKay correspondence, to the exceptional Dynkin diagrams $E_{6}$, $E_{7}$, and $E_{8}$ of the ADE classification \cite{dechant2018trinity}.}
\label{Platonic}
\end{figure}
\end{center}
In Riemannian geometry, the significance of these groups is realized through the spherical $3$-manifolds. By the Killing-Hopf theorem \cite[\S 2.4]{wolf2011spaces}, a complete, simply connected Riemannian manifold of constant positive curvature $+1$ is isometric to the $3$-sphere $S^{3} \cong {\rm{SU}}(2)$. Consequently, any closed orientable spherical $3$-manifold is isometric to a quotient ${\rm{SU}}(2)/\Gamma_{ADE}$, where $\Gamma_{ADE} \subset {\rm{SU}}(2)$ is a finite subgroup and therefore one of the binary polyhedral groups from the ADE list \cite{scott1983geometries}, \cite{milnor1975dimensional}. These manifolds form a crucial class within Thurston's geometrization program \cite{thurston2022geometry}, \cite{thurston2022collected}. The geometrization conjecture, proven by Perelman \cite{perelman2002entropy, perelman2003ricci, perelman2003finite}, asserts that every closed $3$-manifold can be decomposed into pieces, each admitting one of eight geometric structures. The spherical manifolds are precisely those that admit a metric of constant positive curvature, serving as the most symmetric models in this classification.

In a parallel development, flat ${\rm{SU}}(2)$-bundles over closed surfaces ($g >1$) have become a central object of study in differential geometry and mathematical physics, e.g. \cite{hitchin1987self}, \cite{atiyah1983yang}, \cite{sengupta1998moduli}. By the Riemann-Hilbert correspondence, these bundles are classified by conjugacy classes of representations of the surface's fundamental group into ${\rm{SU}}(2)$, see for instance \cite[Chap. 13]{taubes2011differential}. 
The resulting moduli space $\mathcal{M}_{\text{flat}}(\Sigma,{\rm{SU}}(2))$, given by
\begin{equation}
 \mathcal{M}_{\text{flat}}(\Sigma,{\rm{SU}}(2)) \cong {\rm{Hom}}(\pi_{1}(\Sigma),{\rm{SU}}(2))/{\rm{SU}}(2),
\end{equation}
is a finite-dimensional singular symplectic space, whose geometry is deeply intertwined with gauge theory, integrable systems, and geometric quantization, see for instance \cite{dey2006geometric} and references therein. In the above setting, for every ADE subgroup $\Gamma_{ADE} \subset {\rm{SU}}(2)$, we can construct a flat principal ${\rm{SU}}(2)$-bundle $P \to \Sigma_{g}$ ($g > 1$) equipped with a principal flat connection $\theta \in \Omega^{1}(P,\mathfrak{su}(2))$, such that 
\begin{equation}
{\rm{Hol}}(\theta) \cong \Gamma_{ADE}.
\end{equation}
The principal flat ${\rm{SU}}(2)$-bundle $(P,\theta)$ is obtained from a suitable representation $\varrho \colon \pi_{1}(\Sigma_{g}) \to {\rm{SU}}(2)$.

In this work, we explore an unexpected link between spherical space forms and flat ${\rm{SU}}$(2)-bundles over negatively curved surfaces. The central idea is to equip a flat ${\rm{SU}}(2)$-bundle $P$ over a closed hyperbolic surface $(\Sigma,h_{\Sigma})$ with a family of Kaluza-Klein-type metrics 
\begin{equation}
\label{seqKK}
g_{P,n} = \frac{1}{n}\pi^*(h_\Sigma) + (\theta,\theta)_{{\rm{SU}}(2)}, \ \ \ n \in \mathbbm{Z}.
\end{equation}
These metrics are the sum of the pullback of a scalar multiple of the hyperbolic metric on $\Sigma$ (controlling the horizontal directions) and the bi-invariant metric $(-,-)_{{\rm{SU}}(2)}$ on ${\rm{SU}}(2)$ applied to the flat connection form $\theta$ (controlling the vertical directions). By introducing the above sequence of metrics that progressively shrink the horizontal directions, we observe a remarkable collapsing phenomenon. As $n\to\infty$, the diameter of the hyperbolic base shrinks to zero, while the ${\rm{SU}}(2)$ fibers remain unchanged. The Gromov-Hausdorff limit \cite{brin2001course} of this family is precisely the homogeneous spherical $3$-manifold ${\rm{SU}}(2)/\Gamma_{ADE}$ associated with the holonomy group of the flat connection.

Thus, flat ${\rm{SU}}(2)$-bundles over hyperbolic surfaces serve as geometric interpolators between two geometries of opposite curvature signs: the negative curvature of the hyperbolic base ($K=-1$) and the positive curvature of the spherical fiber ($K=+1$). The controlled collapse erases the hyperbolic geometry, revealing the underlying spherical geometry of the quotient. In this setting, our main theorem shows that every orientable spherical $3$-manifold arises as a boundary point in the Gromov--Hausdorff closure of the superspace (\cite{fischer1970theory}, \cite{edwards1975structure}) of Riemannian metrics associated with such a flat bundle. Moreover, the rate of this convergence is quantitatively controlled by the order of the fundamental group of the spherical $3$-manifold and the metric invariant provided by the ratio between the area and the systole (e.g. \cite{gromov1983filling}, \cite{katz2007systolic}) of the hyperbolic base. This establishes a direct geometric link between the topological realization of Thurston's spherical geometries and the systolic optimization of the moduli space of Riemann surfaces \cite{schmutz1993reimann}, \cite{buser2010geometry}, \cite{buser1994period}, \cite{katz2007logarithmic}, offering a dynamical perspective on the emergence of ADE-classified 3-manifolds.

\subsection*{Main Results} In order to state our main result, let us recall some basic generalities on Gromov-Hausdorff spaces and superspaces of Riemannian metrics. 

The Gromov-Hausdorff space $(\mathfrak{M},d_{GH})$ is defined as a metric space consisting of isometry classes of compact metric spaces equipped with the Gromov-Hausdorff distance $d_{GH}$, such that 
\begin{equation}
d_{GH}((X,d_{X}),(Y,d_{Y})):= \inf_{Z}\inf_{f,g}d_{H}^{Z}(f(X),g(Y)),
\end{equation}
where the infimum is taken over all metric spaces $(Z,d_{Z})$ and all isometric embeddings
\begin{center}
$f \colon (X,d_{X}) \to (Z,d_{Z})$ \ \ \ \ and \ \ \ \ $g \colon (Y,d_{Y}) \to (Z,d_{Z})$. 
\end{center}
Here $d_{H}^{Z}$ denotes the Hausdorff distance in $(Z,d_{Z})$, see for instance \cite{gromov1999metric}, \cite{brin2001course}, \cite{ivanov2015gromov}. 

Given a compact connected manifold $M$, denoting by $\mathcal{R}(M)$ the space of all Riemannian metrics on $M$, we can realize the moduli space of Riemannian metrics on $M$
\begin{equation}
\mathcal{S}(M):=\mathcal{R}(M)/{\rm{Diff}}(M),
\end{equation}
also known as superspace \cite{fischer1970theory}, \cite{edwards1975structure}, as a subset of $(\mathfrak{M},d_{GH})$ through the following injective map  
\begin{equation}
\widetilde{\Phi} \colon \mathcal{S}(M) \to \mathfrak{M}, \ \ [g] \mapsto (M,d_{g}), 
\end{equation}
where $d_{g}$ denotes the distance induced by $g$. Considering the Fréchet manifold $\mathcal{R}(M)$, see for instance \cite{tuschmann2015moduli}, one can equip $\mathcal{S}(M)$ with the quotient topology and show that $\widetilde{\Phi}$ is an injective continuous map. Hence, under the identification of sets 
\begin{equation}
\mathcal{S}(M) \cong \big \{ (M,d_{g})\ | \ g \in \mathcal{R}(M)\big \},
\end{equation}
we can study the boundary of $\mathcal{S}(M)$, as subset of $\mathfrak{M}$, by means of sequences of classes of isometric metrics in $\mathcal{S}(M)$. In this setting, the main result of this paper is the following.
\begin{thm}
\label{mainresult}
Let $(S,g_{S})$ be a homogeneous spherical $3$-manifold. Then, there exist a closed orientable hyperbolic surface $\Sigma \in \mathcal{M}_{g}$ and $[P,\theta] \in \mathcal{M}_{{\text{flat}}}(\Sigma,{\rm{SU}}(2))$, such that 
\begin{equation}
(S,d_{g_{S}}) \in \overline{\mathcal{S}(P)}^{{\rm{GH}}} = \overline{\big \{ (P,d_{g}) \  | \ g \in \mathcal{R}(P)\big \}}^{{\rm{GH}}} ,
\end{equation}
where $\overline{\mathcal{S}(P)}^{{\rm{GH}}}$ denotes the Gromov-Hausdorff closure of the superspace $\mathcal{S}(P)$. Moreover, we have 
\begin{equation}
d_{{\rm{GH}}}\big ( (S,d_{g_{S}}),(P,d_{g_{P}})  \big) \leq \frac{1}{2}\frac{|\pi_{1}(S)|{\rm{Area}}(\Sigma)}{{\rm{sys}}_{1}(\Sigma,h_{\Sigma})},
\end{equation}
where $h_{\Sigma}$ is a hyperbolic metric on $\Sigma$ and $g_{P} = \pi^{\ast}(h_{\Sigma}) + (\theta,\theta)_{{\rm{SU}}(2)}$. 
\end{thm}
Here we consider the systole of a nonsimply connected compact Riemannian manifold $(M,g_{M})$. This metric invariant defined by
\begin{equation}
{\rm{sys}}_{1}(M,g_{M}):= \inf\big\{ \ell(\gamma) \ |\ 
\gamma \ \text{is a non-contractible closed curve in } (M,g_{M}) \big\},
\end{equation}
where $\ell(\gamma)$ denotes the length of $\gamma$, see for instance \cite{gromov1983filling}. As we see, Theorem \ref{mainresult} characterizes every homogeneous spherical 3-manifold $(S,g_{S})$ as the limit of a controlled collapsing phenomenon. By equipping a flat ${\rm{SU}}(2)$-bundle $P \to \Sigma$ over a negatively curved base with a sequence of Kaluza-Klein metrics (see Eq. (\ref{seqKK})) the mechanism progressively rescales the horizontal distribution toward zero. As the diameter of the hyperbolic base vanishes, the underlying ADE-classified quotient structure $S \cong {\rm{SU}}(2)/\Gamma_{ADE}$ is revealed, effectively realizing the spherical manifold as a limit of geometries exhibiting both positive and negative curvature.

Furthermore, the theorem provides a precise systolic bound for this convergence, demonstrating that the error in the Gromov-Hausdorff distance is explicitly controlled by the ratio between the area and the systole of the hyperbolic base $(\Sigma,h_{\Sigma})$. This establishes a quantitative link between the topological realization of Thurston's homogeneous spherical geometries and the metric optimization of the moduli space of Riemann surfaces as follows. From Theorem \ref{mainresult}, for a closed hyperbolic surface $\Sigma$ of genus $g \geq 2$, we have the inequality
    \begin{equation}
        \inf_{[P,\theta] \in \mathcal{M}_{{\text{flat}}}(\Sigma,{\rm{SU}}(2))} d_{{\rm{GH}}}\big ( (S,d_{g_{S}}),(P,d_{g_{P}})\big) \leq \frac{1}{2}\,\frac{|\pi_{1}(S)|\,{\rm{Area}}(\Sigma)}{{\rm{sys}}_{1}(\Sigma,h_{\Sigma})}.
    \end{equation}
Here we take the infimum over all flat bundles $[P,\theta]$. Since $\Sigma$ is equipped with a hyperbolic metric $h_{\Sigma}$, the Gauss-Bonnet theorem gives ${\rm{Area}}(\Sigma) = 4\pi(g-1)$. Substituting this into the inequality yields
\begin{equation}
\label{taketheinfimum}
       \inf_{[P,\theta] \in \mathcal{M}_{{\text{flat}}}(\Sigma,{\rm{SU}}(2))} d_{{\rm{GH}}}\big ( (S,d_{g_{S}}),(P,d_{g_{P}})\big) \leq \frac{2\pi|\pi_{1}(S)|(g-1)}{{\rm{sys}}_{1}(\Sigma,h_{\Sigma})}.
\end{equation}
Considering 
\begin{equation}
\label{supsystdef}
{\rm{sys}}_{\max}(g) := \sup_{\Sigma\in\mathcal{M}_{g}}{\rm{sys}}_{1}(\Sigma,h_{\Sigma}),
\end{equation}
and taking the infimum over all hyperbolic surfaces $\Sigma\in\mathcal{M}_{g}$ in Eq. (\ref{taketheinfimum}), we obtain the following corollary.
\begin{corollarynon}
\label{corollaryA}
Under the hypotheses of the last theorem, we have
\begin{equation}
\inf_{\Sigma \in \mathcal{M}_{g}} \Big [\inf_{[P,\theta] \in \mathcal{M}_{{\text{flat}}}(\Sigma,{\rm{SU}}(2))} d_{{\rm{GH}}}\big ( (S,d_{g_{S}}),(P,d_{g_{P}})\big) \Big ]\leq \frac{2\pi |\pi_{1}(S)|(g-1)}{{\rm{sys}}_{\max}(g)}.
\end{equation}
\end{corollarynon}

The geometric significance of Corollary 1.2 becomes apparent when viewed through the lens of systolic geometry \cite{gromov1983filling}, \cite{katz2007systolic}. As we have seen, by the Gauss-Bonnet theorem, the area of the hyperbolic base is a topological constant. Consequently, minimizing the upper bound of the Gromov-Hausdorff distance in Theorem 1.1 reduces precisely to maximizing the function ${\rm{sys}}_{1} \colon \mathcal{M}_{g} \to \mathbbm{R}_{>0}$ over the moduli space of Riemann surfaces. The problem of maximizing the systole function has a long and rich history, deeply intertwined with the theory of lattice packings \cite{conway2013sphere}, \cite{schmutz1993reimann}, spectral geometry \cite{buser2010geometry}, \cite{jenni1984ersten}, and arithmetic groups~\cite{buser1994period}, \cite{katz2007logarithmic}.

Due to Mumford’s generalization of Mahler’s compactness theorem, see \cite{mumford1971remark}, the systole function reaches a global maximum value at some point of the moduli space $\mathcal{M}_{g}$ for every $g > 1$, i.e., the supremum described in Eq. (\ref{supsystdef}) is a maximum. Moreover, from \cite[Lemma 5.2.1]{buser2010geometry}, we have the upper bound
\begin{equation}
\label{upperboundsys}
{\rm{sys}}_{1}(\Sigma) \leq 2\log(4g - 2).
\end{equation}
The exact value of ${\rm{sys}}_{\max}(g)$ is known for $g = 2$; the global maximum is attained by the Bolza surface, see \cite{jenni1984ersten} and \cite{schmutz1993reimann}. For higher genera ($g > 2$), the explicit value of ${\rm{sys}}_{\max}(g)$ is unknown. For results concerning the local maxima of the systole function, we refer the reader to \cite{bourque2021local} and references therein.

In view of the inequality given in Eq. (\ref{upperboundsys}), Buser-Sarnak \cite{buser1994period} constructed a sequence of arithmetic surfaces $\Sigma_{g_{p}}$ of genus $g_{p}$, with
\begin{equation}
g_{p} = (p^{3} - p)\nu + 1,
\end{equation}
such that $p$ is any odd prime number and $\nu$ depends on a division quaternion algebra $A$, satisfying
\begin{equation}
{\rm{sys}}_{1}(\Sigma_{g_{p}}) \geq \frac{4}{3}\log(g_{p}) + C, 
\end{equation}
where $C$ is a positive constant depending only on $A$. Combining the above result with Theorem \ref{mainresult}, we have the following corollary.
\begin{corollarynon}
Let $(S,g_{S})$ be a homogeneous spherical $3$-manifold. Then, there exist a closed orientable hyperbolic surface $\Sigma$ of genus $g(\Sigma) > 1$ and $[P,\theta] \in \mathcal{M}_{{\text{flat}}}(\Sigma,{\rm{SU}}(2))$, such that 
\begin{equation}
d_{{\rm{GH}}}\big ( (S,d_{g_{S}}),(P,d_{g_{P}})  \big) \leq \frac{6\pi|\pi_{1}(S)|(g(\Sigma) - 1)}{4\log(g(\Sigma)) + 3C},
\end{equation}
where $g_{P} = \pi^{\ast}(h_{\Sigma}) + (\theta,\theta)_{{\rm{SU}}(2)}$, such that $h_{\Sigma}$ is a hyperbolic metric on $\Sigma$, and $C > 0$ is constant.
\end{corollarynon}

In order to illustrate explicitly the mechanism provided by Theorem \ref{mainresult}, we show that the Bolza surface yields the optimal error bound for the convergence toward the Poincaré homology sphere \cite{kirby1979eight}.

\subsubsection*{Acknowledgments.} E. M. Correa is supported by S\~{a}o Paulo Research Foundation FAPESP grant 25/18843-1.

\section{Generalities on Hyperbolic Surfaces}

In what follows, a surface is a connected, orientable two-dimensional smooth manifold, without boundary unless otherwise specified.

We start by recalling the following fundamental result \cite[\S 2.4]{wolf2011spaces}.

\begin{theorem}[Killing-Hopf] Any complete and connected Riemann space $(M^{n},g)$, $n \geq 2$, of constant curvature $K$ has a universal cover $S^{n}$, ${\bf{R}}^{n}$ and ${\bf{H}}^{n}$. More precisely, $(M^{n},g)$ is isometric to one of the following quotient spaces:
\begin{enumerate}
\item[(i)] $S^{n}/\Gamma$ with $\Gamma \subset {\rm{O}}(n+1)$, if $K > 0$,
\item[(ii)] ${\bf{R}}^{n}/\Gamma$ with $\Gamma \subset {\rm{E}}(n)$, if $K = 0$,
\item[(iii)] ${\bf{H}}^{n}/\Gamma$ with ${\rm{O}}^{1}(n+1)$, if $K < 0$,
\end{enumerate}
where $\Gamma$ acts freely and properly discontinuously.
\end{theorem}

Given a closed Riemann surface $(\Sigma,g)$, from the Gauss-Bonnet formula 
\begin{equation}
\frac{1}{2\pi}\int_{\Sigma} Kd\mu = \chi(\Sigma),
\end{equation}
the number $r$ defined by the average of its scalar curvature, i.e., 
\begin{equation}
r := \Big (\int_{\Sigma} Kd\mu \Big )\Big / \Big ( \int_{\Sigma}d\mu\Big ),
\end{equation}
is completely determined by the Euler characteristic $\chi(\Sigma)$ of the surface, hence is independent of the metric $g$. In this setting, we have the following important result \cite[Chapter 5]{chow2004ricci}.
\begin{theorem}
If $(\Sigma,g_{0})$ is a closed Riemann surface, there exists a unique solution $g(t)$ of the normalized Ricci-flow
\begin{equation}
\begin{cases} \displaystyle \frac{\partial}{\partial t} g = (r - K)g, \\
g(0) = g_{0}.\end{cases}
\end{equation}
The solution exists for all time. As $t \to +\infty$, the metrics $g(t)$ converge uniformly in any $C^{k}$-norm to a smooth metric $g_{\infty}$ of constant curvature.
\end{theorem}
In this paper, we are interested in the following class of surfaces.
\begin{definition}
We say that a closed surface $\Sigma$ is hyperbolic if $\chi(\Sigma) < 0$.
\end{definition}

From the previous results we have the following corollary.
\begin{corollary}
A closed surface $\Sigma$ is hyperbolic if and only if it admits a Riemannian metric with curvature $K = -1$.
\end{corollary}

\begin{remark}
We shall denote by $\mathcal{M}_{g}$ the moduli space of closed Riemann surfaces of genus $g$ up to biholomorphism. In the particular case that $g \geq 2$, we have the following characterization from the uniformization theorem
\begin{equation}
\mathcal{M}_{g} = \frac{\{\text{hyperbolic surfaces of genus} \ g\}}{\{\text{isometries}\}}.
\end{equation}
In this work we shall denote a point of $\mathcal{M}_{g}$ in an interchangeable way as $\Sigma$ and $(\Sigma,h_{\Sigma})$.
\end{remark}

Given a closed hyperbolic surface $\Sigma$, denoting by $g(\Sigma)$ its genus, since $2 - 2g(\Sigma) = \chi(\Sigma) < 0$, it follows that $g(\Sigma) \geq 2$. In particular, we have the following result (e.g. \cite{jost2006compact}).

\begin{theorem}
If $\Sigma$ is a closed hyperbolic surface, then $\pi_{1}(\Sigma)$ possesses the following presentation 
\begin{equation}
\pi_{1}(\Sigma) = \Big \langle a_{1},b_{1},\ldots,a_{g(\Sigma)},b_{g(\Sigma)} \ \big | \ \prod_{i = 1}^{g(\Sigma)}[a_{i},b_{i}] = 1 \Big \rangle.
\end{equation}
In particular, $\pi_{1}(\Sigma)$ is a non-abelian group.
\end{theorem}

In this paper we shall also consider the following result which encompasses some important facts.

\begin{lemma}
\label{eulercovering}
Let $\Sigma$ be a closed surface of genus $g(\Sigma) > 0$ and let $p \colon \Sigma' \to \Sigma$ be an $n$-sheeted topological covering ($n > 1$). Then, the following hold:
\begin{enumerate}
\item[(1)] $\Sigma'$ is a closed surface.
\item[(2)] $\chi(\Sigma') = n \chi(\Sigma)$ and $g(\Sigma')  = n(g(\Sigma) - 1) + 1$.
\end{enumerate}
In particular, if $\Sigma$ is hyperbolic, so is $\Sigma'$.
\end{lemma}
\begin{proof}
(1) Since $p \colon \Sigma' \to \Sigma$ is a local homeomorphism, it follows that $\Sigma'$ admits a structure of complex manifold, such that $p \colon \Sigma' \to \Sigma$ is a holomorphic map. Moreover, since $\Sigma$ is a compact, connected and orientable manifold without boundary, and the number of sheets of $p$ is finite, it follows that $\Sigma'$ is a closed surface.

(2) Considering the holomorphic map between closed surfaces $p \colon \Sigma' \to \Sigma$, it follows from the general Riemann–Hurwitz (\cite{miles2016riemann}, \cite{jost2006compact}) formula that 
\begin{equation}
\chi(\Sigma') = n\chi(\Sigma) + \sum_{x \in \Sigma'}(m_{x}(p) - 1),
\end{equation}
where $m_{x}(p)$ is the ramification index at $x \in \Sigma'$. Since a topological covering of a Riemann surface is unramified covering, we conclude that $\chi(\Sigma') = n\chi(\Sigma)$ and $g(\Sigma')  = n(g(\Sigma) - 1) + 1$. In particular, if $g(\Sigma) > 1$, then $g(\Sigma') > 1$, i.e., if $\Sigma$ is hyperbolic, so is $\Sigma'$.
\end{proof}

\begin{center}
\begin{figure}[H]
\centering\includegraphics[scale = .12]{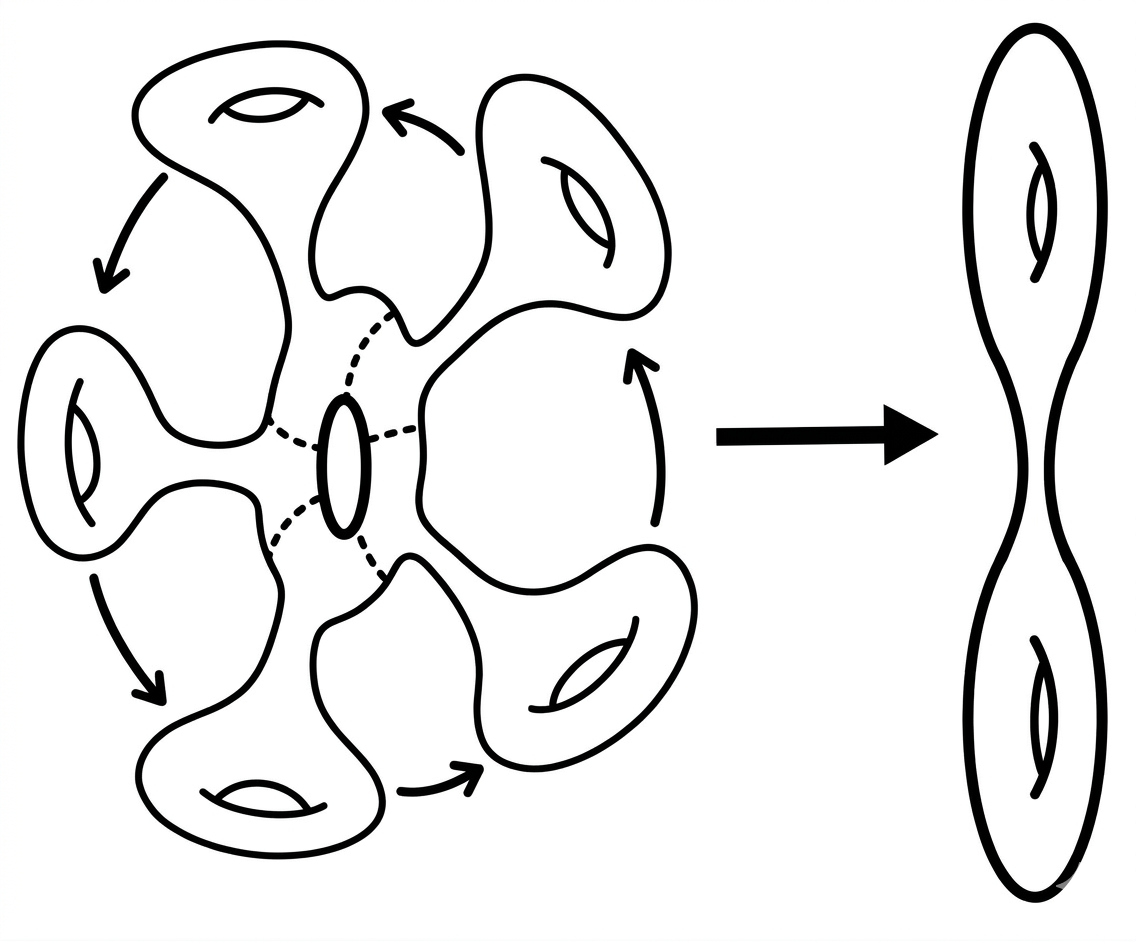}
\caption{The arrows above indicate a map which generates the group of deck transformations $\mathbbm{Z}_{5}$. In this case, we have a covering $\Sigma_{6} \to \Sigma_{2}$, such that $\Sigma_{2} = \Sigma_{6}/\mathbbm{Z}_{5}$. }
\label{n-covering_surface}
\end{figure}
\end{center}

In order to prove our main result, it will be important to consider the following result (see \cite{chang1977diameters}).

\begin{theorem}
\label{diametersystole}
Let $\Sigma$ be a closed hyperbolic surface equipped with a hyperbolic metric $h_{\Sigma}$. Then, we have the following inequality
\begin{equation}
{\rm{diam}}(\Sigma,h_{\Sigma}) \leq \frac{{\rm{Area(\Sigma)}}}{{\rm{sys}}_{1}(\Sigma,h_{\Sigma})},
\end{equation}
such that ${\rm{sys}}_{1}(\Sigma,h_{\Sigma})$ is the systole of $\Sigma$, i.e., the least length of a non-contractible loop in $\Sigma$.
\end{theorem}

\begin{remark}
\label{diamsystoleineq}
Let $p \colon \Sigma' \to \Sigma$ be an $m$-sheeted topological covering ($m > 1$) of a hyperbolic surface. In this setting, we have 
\begin{equation}
{\rm{diam}}(\Sigma',p^{\ast}(h_{\Sigma})) \leq \frac{m{\rm{Area}}(\Sigma)}{{\rm{sys}}_{1}(\Sigma,h_{\Sigma})}.
\end{equation}
In fact, since every non-contractible loop in $\Sigma'$ is projected to a non-contractible loop in $\Sigma$ with the same length, we obtain the desired inequality. 

\end{remark}
\section{Flat Principal Bundles}

In this section, we introduce some basic results and generalities related to flat connections on principal bundles.

A smooth fiber bundle $\pi \colon P \to M$ with fiber $G$ is a principal $G$-bundle if $G$ acts smoothly and freely on $P$ on the right and the fiber-preserving local trivializations
\begin{equation}
\phi_{U} \colon \pi^{-1}(U) \to U \times G,
\end{equation}
are $G$-equivariant, where $G$ acts on $U \times G$ on the right by $(x,h) \cdot g = (x,hg)$. Denoting by $\mathfrak{g}$ the Lie algebra of $G$, we have the following definition.
\begin{definition}
A principal connection on a principal $G$-bundle $G \hookrightarrow P \to M$ is a $\mathfrak{g}$-valued $1$-form $\theta \colon TP \to \mathfrak{g}$ on $P$, satisfying the following
\begin{equation}
\label{connection}
R_{g}^{\ast}\theta = {\rm{Ad}}(g^{-1})\theta, \ \forall g \in G, \ \ \   \theta(X_{\ast}) = X, \ \forall X \in \mathfrak{g},
\end{equation}
where $X_{\ast}$ is the Killing vector field generated by $X \in \mathfrak{g}$, and $R_{g}(u) = ug$, $\forall u \in P$ and $\forall g \in G$.
\end{definition}

In the above setting, we have the decomposition 
\begin{equation}
\label{decompositionconnection}
TP = H(P) \oplus V(P),
\end{equation}
such that $H(P) = \ker(\theta)$ (horizontal subbundle) and, $\forall u \in P$, $V(P)_{u} = j_{u \ast}(\mathfrak{g})$, where $j_{u} \colon G \to P$ is the map defined by 
\begin{equation}
j_{u}(g) := ug, \ \ \ \forall g \in G. 
\end{equation}
Moreover, from Eq. (\ref{connection}), we have
\begin{equation}
H(P)_{ug} = R_{g\ast}H(P)_{u},
\end{equation}
$\forall u \in P$ and $\forall g \in G$. In this last setting, a curve $\gamma \colon I \to P$ is said to be $\theta$-horizontal if 
\begin{equation}
\dot{\gamma}(t) \in H(P)_{\gamma(t)}, \ \ \forall t \in I.
\end{equation}
Given a principal $G$-bundle $G \hookrightarrow P \to M$, fixing a principal connection $\theta \in \Omega^{1}(P,\mathfrak{g})$, we have the following well-known results.

\begin{proposition}
Let $\alpha \colon [0,1] \to M$ be a smooth curve and let $X \colon U \to TM$ be a smooth vector field, such that $U\subseteq M$ is an open set. Then, the following holds:
\begin{enumerate}
\item[(i)] For every $u \in \pi^{-1}(\{\alpha(0)\})$, there exists a unique smooth curve $\alpha_{u}^{h} \colon [0,1] \to P$, such that 
\begin{equation}
\alpha_{u}^{h}(0) = u, \ \ \ \dot{\alpha}_{u}^{h}(t) \in H(P)_{\alpha_{u}^{h}(t)}, \ \ \text{and} \ \ \pi(\alpha_{u}^{h}(t)) = \alpha(t), \ \ \forall t \in [0,1];
\end{equation}
\item[(ii)] There exists a unique smooth vector field $X^{H} \colon \pi^{-1}(U) \to TP$, such that 
\begin{equation}
X^{H}(u) \in H(P)_{u}, \ \ \ \ \text{and} \ \ \ \ \pi_{\ast}(X^{H}(u))=X(\pi(u)), \ \ \forall u \in \pi^{-1}(U).
\end{equation}
\end{enumerate}
\end{proposition}

\begin{remark}
For the proof of the above results, see \cite[\S 1.7]{rudolph2017differential} and \cite[Proposition 1.2]{kobayashi1963foundations}.
\end{remark}

From item (i) of the above proposition, we obtain the following result.
 
\begin{proposition}
Let $\alpha \colon [0,1] \to M$ be a piecewise smooth curve. Then, the map
\begin{equation}
\tau_{\alpha} \colon \pi^{-1}(\{\alpha(0)\}) \to \pi^{-1}(\{\alpha(1)\}), \ \ \ \tau_{\alpha}(u)= \alpha_{u}^{h}(1), \ \ \forall u\in \pi^{-1}(\{\alpha(0)\}),
\end{equation}
is a $G$-equivariant diffeomorphism.
\end{proposition}
\begin{remark}
For more details on the above result, see \cite[\S 5.8]{hamilton2017mathematical} and \cite[Chapter 3, \S 11]{rudolph2017differential}.
\end{remark}
\begin{definition}
$\tau_{\alpha}$ is called the parallel transport along $\alpha$ with respect to the connection $\theta$.
\end{definition}
Now we define an equivalence relation $\sim$ on $P$ by saying that, $\forall u,v \in P$, $u \sim v$ if and only if $u$ and $v$ can be joined by a piecewise smooth $\theta$-horizontal path in $P$. If $\alpha$ is a loop based at $m \in M$, then, for each $ u \in \pi^{-1}(\{m\})$, there exists ${\rm{hol}}_{\theta,\alpha}(u) \in G$ such that 
\begin{equation}
u \sim \tau_{\alpha}(u) = u{\rm{hol}}_{\theta,\alpha}(u). 
\end{equation}
From above, fixing $u \in  \pi^{-1}(\{m\})$, we obtain a map $\mathcal{L}_{m}M \to G$, such that 
\begin{equation}
\alpha \mapsto {\rm{hol}}_{\theta,\alpha}(u), \ \ \forall \alpha \in \mathcal{L}_{m}M,
\end{equation}
where $\mathcal{L}_{m}M$ is the set of all piecewise smooth loops based at $m \in M$.
\begin{definition}
The holonomy group of $\theta$ based at $u \in P$ is defined by
\begin{equation}
{\rm{Hol}}_{u}(\theta) := \big \{ {\rm{hol}}_{\theta,\alpha}(u) \in G \ \  | \alpha \in \mathcal{L}_{m}M\big \} = \big \{ g \in G \ \  | \ \ u \sim ug\big \}.
\end{equation}
The restricted holonomy group of $\theta$ based at $u$ is defined by the subgroup ${\rm{Hol}}_{u}^{0}(\theta)$ of parallel transports along contractible loops based at $m=\pi(u) \in M$.
\end{definition}

The holonomy group ${\rm{Hol}}_{u}(\theta) \subset G$, $u \in P$, is a Lie subgroup\footnote{ Recall that $H \subset G$ is a Lie subgroup if $H$ is an immersed submanifold of $G$, such that the product $H \times H \to H$ is differentiable with respect to the intrinsic structure of $H$.} of $G$ whose connected component of the identity is given by the connected Lie subgroup ${\rm{Hol}}_{u}^{0}(\theta) \subset G$, which is called reduced holonomy group of $\theta$. Moreover, we have the following properties:
\begin{enumerate}
\item[(1)] ${\rm{Hol}}_{ug}(\theta) = g^{-1}{\rm{Hol}}_{u}(\theta)g$, $\forall g \in G$,
\item[(2)] Denoting $m = \pi(u)$, we have an epimorphism $\varrho_{\theta} \colon \pi_{1}(M,m) \to {\rm{Hol}}_{u}(\theta)/{\rm{Hol}}_{u}^{0}(\theta)$, such that 
\begin{equation}
\varrho_{\theta}([\alpha]) = {\rm{hol}}_{\theta,\alpha}(u) {\rm{Hol}}_{u}^{0}(\theta), \ \ \ \forall [\alpha] \in \pi_{1}(M,m)
\end{equation}
\item[(3)] If $M$ is simply connected, then ${\rm{Hol}}_{u}(\theta) = {\rm{Hol}}_{u}^{0}(\theta)$, $\forall u \in P$.
\end{enumerate}
\begin{remark}
Given $u,v \in P$, since ${\rm{Hol}}_{v}(\theta) = g^{-1}{\rm{Hol}}_{u}(\theta)g$, for some $g \in G$, we can also denote the holonomy group of $\theta$ by ${\rm{Hol}}(\theta)$ without referring to the base point.
\end{remark}

\begin{definition}
A reduction of a principal $G$-bundle $\pi \colon P \to M$ to a Lie subgroup $H$ is defined by a principal $H$-bundle $\pi_{Q} \colon Q \to M$ together with a $H$-equivariant smooth map $\Psi \colon Q \to P$, which covers the identity map ${\rm{id}}_{M} \colon M \to M$.
\end{definition}

Given a principal $G$-bundle $\pi \colon P \to M$, fixing a connection $\theta \in \Omega^{1}(P,\mathfrak{g})$, for every $u \in P$, let us consider the following subset 
\begin{equation}
\label{Holonomybundle}
P_{u}(\theta): = \big \{ v \in P \ \ | \ \ u \sim v\big \},
\end{equation}
i.e., the set of points which can be joined to $u$ by a $\theta$-horizontal path. 
\begin{theorem}
\label{holonomybundle}
In the last setting, the following hold:
\begin{enumerate}
\item[(a)] $P_{u}(\theta)$ is a principal ${\rm{Hol}}_{u}(\theta)$-bundle over $M$;
\item[(b)] The restriction of $\theta$ to $P_{u}(\theta)$ defines a principal connection.
\end{enumerate}
\end{theorem}
\begin{remark}
For the proof of the above result, see \cite[\S 1.7]{rudolph2017differential} and \cite[Chapter 2, \S 7]{kobayashi1963foundations}.
\end{remark}

\begin{remark}
Notice that $P_{u}(\theta) \hookrightarrow P$ is a reduction of $\pi \colon P \to M$ to ${\rm{Hol}}_{u}(\theta) \subset G$. The principal bundle $P_{u}(\theta)$ is called the holonomy bundle of $\theta$ through $u \in P$.
\end{remark}

\begin{definition}
Let $P$ be a principal $G$-bundle over $M$, and let $\theta$ be a principal connection on $P$. We define the curvature $2$-form $D\theta \in \Omega^{2}(P,\mathfrak{g})$ of the connection $\theta$ as
\begin{equation}
D\theta(X,Y) := {\rm{d}}\theta (\Pi(X),\Pi(Y)),
\end{equation}
for all $X,Y \in \mathfrak{X}(P)$, such that $\Pi \colon TP \to H(P)$ is the projection onto the horizontal subbundle $H(P)$.
\end{definition}

\begin{theorem}[Ambrose-Singer]
\label{AS-theo}
Let $P$ be a principal $G$-bundle over $M$, and let $\theta$ be a principal connection on $P$. Then, for every $u \in P$, the Lie algebra $\mathfrak{hol}_{u}(\theta)$ of ${\rm{Hol}}_{u}(\theta)$ is given by
\begin{equation}
{\mathfrak{hol}}_{u}(\theta) = {\rm{Span}}\big \{  (D\theta)_{v}\big (V,W\big ) \ \  | \ \ v \in P_{u}(\theta), \ \ V,W \in \ker(\theta)_{v}\big\}.
\end{equation}
\end{theorem}

\begin{remark}
For the proof of the above theorem, we suggest \cite[Theorem 1.7.15]{rudolph2017differential}.
\end{remark}

Now we consider the following definition.

\begin{definition}
\label{flatconnection}
A principal connection $\theta \in \Omega^{1}(P,\mathfrak{g})$ on a principal $G$-bundle $G \hookrightarrow P \to M$ is said to be flat if $D\theta \equiv 0$. In this case, we say that the pair $(P,\theta)$ is a flat principal $G$-bundle.
\end{definition}

\begin{remark}
If $(P,\theta)$ is a flat principal $G$-bundle, it follows from Ambrose-Singer theorem that the reduced holonomy group ${\rm{Hol}}_{u}^{0}(\theta)$ is trivial. Therefore, in this case we have a well-defined homomorphism $\varrho_{\theta} \colon \pi_{1}(M,m) \to G$, such that $m = \pi(u)$ and  
\begin{equation}
\varrho_{\theta}([\alpha]) = {\rm{hol}}_{\theta,\alpha}(u), \ \ \ \forall [\alpha] \in \pi_{1}(M,m).
\end{equation}
In particular, we have ${\rm{Im}}(\varrho_{\theta}) = {\rm{Hol}}_{u}(\theta) \subset G$. Considering the following equivalence relation in the set of homomorphisms ${\rm{Hom}}(\pi_{1}(M,m),G)$:
\begin{equation}
\varrho \sim \varrho' \iff \varrho'(-) = g \varrho(-)g^{-1}, \ \text{for some} \ g \in G,
\end{equation}
and denoting by ${\rm{Hom}}(\pi_{1}(M,m),G)/G$ the set of equivalence classes, since different choices of ``$m$'' change $\pi_{1}(M,m)$ essentially by conjugation, we obtain ${\rm{Hom}}(\pi_{1}(M,m),G)/G = {\rm{Hom}}(\pi_{1}(M),G)/G$. Further, since 
\begin{equation}
{\rm{hol}}_{\theta,\alpha}(vg) = g{\rm{hol}}_{\theta,\alpha}(v)g,
\end{equation}
for some $v \in \pi^{-1}(\{m\})$, it follows that the map $(P,\theta) \mapsto [\varrho_{\theta}]$ is well-defined. 
\end{remark}

\begin{remark}
\label{coveringflatbundle}
From the previous results, given a flat principal $G$-bundle $(P,\theta)$, such that ${\rm{Hol}}_{u}(\theta) \subset G$ is a closed subgroup, it follows that ${\rm{Hol}}_{u}(\theta)$ is a discrete subgroup for all $u \in P$. In particular, since a connected principal bundle with discrete structure group is necessarily non-trivial, if $M$ is connected and $\theta$ is flat, then $P_{u}(\theta)$ is a non-trivial principal bundle over $M$ with discrete structure group, for all $u \in P$. In fact, in this last setting, under the assumption that ${\rm{Hol}}_{u}(\theta) \subset G$ is a closed subgroup and that the connection $\theta$ is flat, it follows that 
\begin{equation}
{\rm{Hol}}_{u}(\theta) \hookrightarrow P_{u}(\theta) \to M,
\end{equation}
is a covering space, for every $u \in P$. 
\end{remark}
In the setting of Definition \ref{flatconnection}, we define the following equivalence relation on the set of flat principal $G$-bundles:
\begin{equation}
(P,\theta) \sim (Q,\omega) \iff \exists \varphi \colon P \to Q \ \ (\text{isomorphism}), \ \ {\text{such that}} \ \theta = \varphi^{\ast}(\omega).
\end{equation}
From above, we have the moduli space of flat principal $G$-bundles over $M$ as 
\begin{equation}
\mathcal{M}_{\text{flat}}(M,G) := \frac{\{ (P,\theta) \  | \ \theta \ \text{is a flat connection on} \ P\}}{\sim}.
\end{equation}
The following result provides a complete classification for flat connections.

\begin{theorem}
\label{classfyingflatSU2bundles}
There exists a 1-1 correspondence 
\begin{equation}
 \mathcal{M}_{\text{flat}}(M,G) \cong {\rm{Hom}}(\pi_{1}(M),G)/G,
\end{equation}
such that $G$ acts on ${\rm{Hom}}(\pi_{1}(M),G)$ by conjugation.
\end{theorem}

\begin{remark}
Considering an equivalence class $[P,\theta] \in \mathcal{M}_{\text{flat}}(M,G)$ the map $[P,\theta] \mapsto [\varrho_{\theta}]$ defines the desired bijective map. The inverse map can be described as follows: given $\varrho \in {\rm{Hom}}(\pi_{1}(M),G)$, we set 
\begin{equation}
P_{\varrho}:= \widetilde{M} \times_{\varrho} G,
\end{equation}
such that $\widetilde{M}$ is the universal covering space of $M$. From above, we obtain a principal $G$-bundle equipped with principal flat connection $\theta_{\varrho}$, such that ${\rm{Hol}}_{u}(\theta_{\varrho}) \cong {\rm{Im}}(\varrho)$, for every $u \in P_{\varrho}$. The desired inverse is obtained from the map $[\varrho] \mapsto [P_{\varrho},\theta_{\varrho}]$ defines the inverse map. For more details, see for instance \cite[p. 159-164]{taubes2011differential}
\end{remark}

\begin{example}
\label{representationfiniteADE}
Let $M = \Sigma$ be a closed surface of genus $g(\Sigma) \geq 2$ and let $G = {\rm{SU}}(2)$. The binary polyhedral subgroups of ${\rm{SU}}(2)$ are defined by the inverse images of the polyhedral subgroups of ${\rm{SO}}(3)$ via the double cover ${\rm{SU}}(2) \to {\rm{SO}}(3)$. These subgroups have a presentation of the form
\begin{equation}
\Gamma_{ADE} := \langle R,S,T \ | \ R^{a} = S^{b} = T^{c} = RST \rangle,
\end{equation}
such that $Z:= RST$ is a central element of order $2$, see for instance \cite{coxeter1940binary}, \cite{coxeter2013generators}. The groups as above are finite if and only if $(a,b,c)$ satisfies 
\begin{equation}
\frac{1}{a} + \frac{1}{b} + \frac{1}{c} > 1.
\end{equation}
From above, the possible solutions are
\begin{equation}
(a,b,c) = (1,1,n), \ (2,2,n), \ (2,3,3), \ (2,3,4) \ \text{and} \ (2,3,5).
\end{equation}
In the particular cases $(a,b,c) = (2,3,3), \ (2,3,4) \ \text{and} \ (2,3,5)$, observing that $R^{2} = RST \Rightarrow R = ST$, by setting $A = S$ and $B = T$, we obtain the following presentation 
\begin{equation}
\Gamma_{ADE} = \langle A,B \ | \ A^{b} = B^{c} = (AB)^{2} \rangle.
\end{equation}
After some change of variables and manipulations we have the following description for all the binary polyhedral subgroups of ${\rm{SU}}(2)$:

\begin{center}
\begin{table}[H]
\centering
\begin{tabular}{ccccc}
\toprule
\textbf{ADE type} & \textbf{Dynkin diagram} & \textbf{Group} & \textbf{Order} & \textbf{Presentation} \\
\midrule
$A_{n-1}$ & \dynkin[edgelength=0.3cm]{A}{2} $\cdots$ \dynkin[edgelength=0.5cm]{A}{3} & Cyclic $C_n$ & $n$ & $\langle A \mid A^n = 1 \rangle$ \\
\midrule
$D_{n+2}$ & \dynkin[edgelength=0.3cm]{A}{2} $\cdots$ \dynkin[edgelength=0.5cm]{D}{4}  & Binary dihedral $D_n^{\ast}$ & $4n$ & $\langle A, B \mid A^{2} = B^{2} = (AB)^{n} \rangle$ \\
\midrule
$E_6$ & \dynkin[edgelength=0.3cm]{E}{6} & Binary tetrahedral $T^{\ast}$ & $24$ & $\langle A, B \mid A^3 = B^3 = (AB)^2 \rangle$ \\
\midrule
$E_7$ & \dynkin[edgelength=0.3cm]{E}{7} & Binary octahedral $O^{\ast}$ & $48$ & $\langle A, B \mid A^3 =B^4 = (AB)^{2}\rangle$ \\
\midrule
$E_8$ & \dynkin[edgelength=0.3cm]{E}{8} & Binary icosahedral $I^{\ast}$ & $120$ & $\langle A, B \mid A^{3} = B^{5} = (AB)^2\rangle$\\
\bottomrule
\end{tabular}
\caption{Presentations of the binary polyhedral subgroups of ${\rm{SU}}(2)$ (ADE classification) with their Dynkin diagrams. These are all the finite subgroups of ${\rm{SU}}(2)$.}
\label{ADEtable}
\end{table}
\end{center}

From above, in terms of generators and relations, we have the following uniform presentation
\begin{equation}
\Gamma_{ADE} = \langle A,B \ | \ \mathscr{R}(A,B) \rangle.
\end{equation}
In the above description if $\Gamma_{ADE} = C_{n}$ (Cyclic), we set $B = I$. Now we recall that, also in terms of generators and relations, we have the following presentation for the fundamental group of $\Sigma$
\begin{equation}
\pi_{1}(\Sigma) = \Big \langle a_{1},b_{1},\ldots,a_{g(\Sigma)},b_{g(\Sigma)} \ \big | \ \prod_{i = 1}^{g(\Sigma)}[a_{i},b_{i}] = 1 \Big \rangle.
\end{equation}
Given a binary polyhedral subgroup $\Gamma_{ADE} \subset {\rm{SU}}(2)$, we define $\varrho \colon \pi_{1}(\Sigma) \to {\rm{SU}}(2)$ by setting
\begin{equation}
\varrho(a_{1}) = A, \ \varrho(b_{1}) = B, \ \varrho(a_{2}) = B, \ \varrho(b_{2}) = A, \ \varrho(a_{i}) = \varrho(b_{i}) = I, \ i = 3,\ldots,g.
\end{equation}
As we see, denoting $C = [A,B]$, by construction we have
\begin{equation}
\varrho \Big ( \prod_{i = 1}^{g(\Sigma)}[a_{i},b_{i}]\Big ) = \prod_{i = 1}^{g(\Sigma)}[\varrho(a_{i}),\varrho(b_{i})] = C C^{-1} = I,
\end{equation}
i.e., $\varrho \colon \pi_{1}(\Sigma) \to {\rm{SU}}(2)$ is a well-defined homomorphism. Further, since ${\rm{Im}}(\varrho) = \Gamma_{ADE}$, from Theorem \ref{classfyingflatSU2bundles}, we have a flat principal ${\rm{SU}}(2)$-bundle $(P_{\varrho},\theta_{\varrho})$ over $\Sigma$, such that ${\rm{Hol}}_{u}(\theta_{\varrho}) \cong \Gamma_{ADE}$, for all $u \in P_{\varrho}$. Following Remark \ref{coveringflatbundle}, we obtain in this case a $|\Gamma_{ADE}|$-sheeted covering
\begin{equation}
\Gamma_{ADE} \hookrightarrow P_{u}(\theta_{\varrho}) \to \Sigma,
\end{equation}
for every $u \in P_{\varrho}$. Further, according to the Riemann-Hurwitz formula for unbranched covers (\cite{miles2016riemann}, \cite{jost2006compact}), see also Lemma \ref{eulercovering}, we have 
\begin{equation}
\begin{split}
& \chi(P_{u}(\theta_{\varrho})) = |\Gamma_{ADE}| \chi(\Sigma), \\
& g(P_{u}(\theta_{\varrho})) = |\Gamma_{ADE}|(g(\Sigma)-1)+1,
\end{split}
\end{equation}
such that $g(P_{u}(\theta_{\varrho}))$ denotes the genus of the Riemann surface $P_{u}(\theta_{\varrho})$, for every $u \in P$.
\end{example} 

\begin{remark}
Given closed surface $\Sigma$, it is worth pointing out the following. Since 
\begin{center}
$H^{1}(\Sigma,\pi_{0}({\rm{SU}}(2))) = H^{2}(\Sigma,\pi_{1}({\rm{SU}}(2))) = \{0\}$, 
\end{center}
it follows that every principal ${\rm{SU}}(2)$-bundle over $\Sigma$ is isomorphic to the trivial bundle $P_{0} = \Sigma \times {\rm{SU}}(2)$.
\end{remark}

\section{Spherical 3-manifolds and Flat ${\rm{SU}}(2)$-bundles}

Consider the following definition \cite[\S 2.4]{wolf2011spaces}.

\begin{definition}
A Riemannian manifold $(M^{n},g)$ is a spherical manifold if it is a complete, connected Riemannian manifold with constant sectional curvature $K=1$.
\end{definition}

Under the identification $S^{3} \cong {\rm{SU}}(2)$, from the Killing-Hopf theorem and Synge's theorem it follows that every spherical $3$-manifold is a orientable compact manifold of the form ${\rm{SU}}(2)/\Gamma$. The following result classify when this quotient is homogeneous.
\begin{theorem}
 If $(M^{3},g)$ is a spherical 3-manifold, then either
\begin{enumerate}
\item[{\bf{(i)}}] $(M^{3},g)$ is isometric to a homogeneous manifold ${\rm{SU}}(2)/\Gamma_{ADE}$, or
\item[{\bf{(ii)}}] $(M^{3},g)$ is isometric to a lens space $L(n,q)$ s.t. $q \not\equiv \pm 1 \ ({\rm{mod}} \ n)$.
\end{enumerate}
\end{theorem}
For more details on the above classification, see for instance \cite{wolf2011spaces}. From above, it follows that a homogeneous spherical $3$-manifold is isometric to a homogeneous ${\rm{SU}}(2)$-space, namely, we have an isometry 
\begin{equation}
\varphi \colon (M^{3},g) \to ({\rm{SU}}(2)/\Gamma_{ADE},<\cdot,\cdot>),
\end{equation}
where $\Gamma_{ADE} \subset {\rm{SU}}(2)$ is a binary polyhedral subgroup given by the ADE classification (Table \ref{ADEtable}) and $<\cdot,\cdot>$ is the Riemannian metric induced by the bi-invariant inner product
\begin{equation}
(X,Y)_{{\rm{SU}}(2)} := - \frac{1}{2}{\rm{tr}}(XY),
\end{equation}
for all $X,Y \in \mathfrak{su}(2) = T_{e}{\rm{SU}}(2)$. That is, considering the normal covering $\pi \colon {\rm{SU}}(2) \to {\rm{SU}}(2) / \Gamma_{ADE}$, the Riemannian metric $<\cdot,\cdot>$ is the unique Riemannian metric on ${\rm{SU}}(2) / \Gamma_{ADE}$ such that 
\begin{equation}
q \colon ({\rm{SU}}(2),(\cdot,\cdot)_{{\rm{SU}}(2)}) \to ({\rm{SU}}(2) / \Gamma_{ADE}, <\cdot,\cdot>),
\end{equation}
is a local isometry, see for instance \cite[p. 60]{wolf2011spaces}. 

In order to establish our next result, let us introduce some concepts. Given a flat principal ${\rm{SU}}(2)$-bundle $(P,\theta)$ over a closed surface $\Sigma$ of genus $g(\Sigma) \geq 2$, we can choose a Riemannian metric $h_{\Sigma}$ on $\Sigma$ and equip the flat principal ${\rm{SU}}(2)$-bundle $(P,\theta)$ with the bundle-type (Kaluza-Klein) Riemannian metric
\begin{equation}
g_{P}:= \pi^{\ast}(h_{\Sigma}) + (\theta,\theta)_{{\rm{SU}}(2)}.
\end{equation}
By keeping the above notation, we have the following result.

\begin{lemma}
\label{submersiondistortion}
In the above setting, if ${\rm{Hol}}(\theta) \cong \Gamma_{ADE}$ for some $\Gamma_{ADE} \subset {\rm{SU}}(2)$, then there exists a proper Riemannian submersion $\Psi \colon (P,g_{P}) \to ({\rm{SU}}(2) / \Gamma_{ADE}, <\cdot,\cdot>)$, such that 
\begin{equation}
\Psi^{-1}(\{\Psi(v)\}) = P_{v}(\theta), \ \ \forall v \in P.
\end{equation}
\end{lemma}
\begin{proof}
Let us assume that ${\rm{Hol}}_{u}(\theta)  = \Gamma_{ADE}$, for some $u \in P$. In order to prove the desired result, we proceed as follows (cf. \cite{correa2025bundle}).

\begin{enumerate}
\item \underline{Existence of $\Psi$}: Since $P_{u}(\theta) \hookrightarrow P$ is a reduction of $\pi \colon P \to M$ to $\Gamma_{ADE} \subset {\rm{SU}}(2)$, it follows that there exists a global section of the associated bundle
\begin{equation}
E := P \times_{{\rm{SU}}(2)}({\rm{SU}}(2)/\Gamma_{ADE}),
\end{equation}
see for instance \cite{kobayashi1963foundations}. Now we notice that 
\begin{equation}
\Gamma^{\infty}(E) \cong C^{\infty}(P;{\rm{SU}}(2)/\Gamma_{ADE})^{{\rm{SU}}(2)}, 
\end{equation}
where $C^{\infty}(P;{\rm{SU}}(2)/\Gamma_{ADE})^{{\rm{SU}}(2)}$ is the subset of smooth maps $F \colon P \to {\rm{SU}}(2)/\Gamma_{ADE}$, such that
\begin{equation}
F(vg) = g^{-1}F(v), \ \ \forall (v,g) \in P \times {\rm{SU}}(2).
\end{equation}
e.g. \cite[Proposition 1.2.6]{rudolph2017differential}. From above, it follows that the reduction to $\Gamma_{ADE} \subset {\rm{SU}}(2)$ is equivalent to the existence of a smooth ${\rm{SU}}(2)$-equivariant map $\Psi \colon P \to {\rm{SU}}(2)/\Gamma_{ADE}$ uniquely determined by the condition
\begin{equation}
\Psi^{-1}(\{e\Gamma_{ADE}\}) = P_{u}(\theta),
\end{equation}
see for instance \cite[Proposition 1.6.2]{rudolph2017differential}. Given any $v \in P$, by taking the $\theta$-horizontal lift $\alpha^{h}$ of some piecewise smooth path $\alpha$ connecting $\pi(u)$ and $\pi(v)$, it follows that 
\begin{equation}
u = \alpha^{h}(0) \sim \alpha^{h}(1) \in \pi^{-1}(\{\pi(v)\}).
\end{equation}
Since ${\rm{SU}}(2)$ acts transitively on the fibers, we have $g \in {\rm{SU}}(2)$, such that $vg = \alpha^{h}(1) \in P_{u}(\theta)$, i.e., 
\begin{equation}
v = \alpha^{h}(1)g^{-1}.
\end{equation}
Therefore, we have 
\begin{equation}
\Psi(v) = \Psi(\alpha^{h}(1)g^{-1}) = g\Psi(\alpha^{h}(1)) = g\Gamma_{ADE}.
\end{equation}
From above, we see that $\Psi \colon P \to {\rm{SU}}(2)/\Gamma_{ADE}$ is surjective. Further, given $X \in \mathfrak{su}(2)$ and $v \in P$, such that $\Psi(v) = g\Gamma_{ADE}$, we obtain
\begin{equation}
\label{derivativesub}
(D\Psi)_{v}(X_{\ast}(v)) = \frac{d}{dt}\Big|_{t=0}\Psi(v\exp(tX)) = \frac{d}{dt}\Big|_{t=0}\exp(-tX)g{\rm{Hol}}_{u}(\theta) = - (Dq)_{g}(X(g)),
\end{equation}
where $q \colon {\rm{SU}}(2) \to {\rm{SU}}(2)/\Gamma_{ADE}$ is the canonical projection. Here we identify $\mathfrak{su}(2)$ with the Lie algebra of right-invariant vector fields on ${\rm{SU}}(2)$. Since $(Dq)_{g}$ is surjective for all $g \in {\rm{SU}}(2)$, we conclude that the map $\Psi \colon P \to {\rm{SU}}(2)/\Gamma_{ADE}$ is a proper submersion. 

\item \underline{$\Psi^{-1}(\Psi(v)) = P_{v}(\theta), \ \ \forall v \in P$}: Given $v,v' \in P$, we have $v = u'a^{-1}$ and $v' = u''b^{-1}$, such that $u',u'' \in P_{u}(\theta)$ and $a,b \in {\rm{SU}}(2)$. Since the action of ${\rm{SU}}(2)$ on $P$ maps each horizontal curve into horizontal curve, we obtain the following
\begin{equation}
v \sim v' \iff u'a^{-1} \sim u''b^{-1} \iff u \sim ub^{-1}a \iff b^{-1}a \in \Gamma_{ADE} = {\rm{Hol}}_{u}(\theta).
\end{equation}
From above, it follows that
\begin{equation}
v \sim v' \Rightarrow \Psi(v) = a\Gamma_{ADE} = b(b^{-1}a)\Gamma_{ADE} = b\Gamma_{ADE} = \Psi(v').
\end{equation}
Conversely, we notice that 
\begin{equation}
\Psi(v) = \Psi(v') \iff a\Gamma_{ADE} = b\Gamma_{ADE} \iff a^{-1}b \in \Gamma_{ADE} = {\rm{Hol}}_{u}(\theta).
\end{equation}
Therefore, we conclude that 
\begin{equation}
\Psi(v) = \Psi(v') \Rightarrow v = u'a^{-1} \sim ua^{-1} \sim  u(a^{-1}b)b^{-1} \sim ub^{-1} \sim u''b^{-1} = v'.
\end{equation}
Hence, we obtain $v \sim v' \iff \Psi(v) = \Psi(v')$, $\forall v,v' \in P$.

\item \underline{$\Psi$ is a Riemannian submersion:} In order to conclude the proof, it remains to show that the map
\begin{equation}
(D\Psi)_{v} \colon (\ker(D\Psi)_{v})^{\perp} \to T_{\Psi(v)}({\rm{SU}}(2)/\Gamma_{ADE}),
\end{equation}
is an isometry for every $v \in P$, here $(\ker(D\Psi)_{v})^{\perp}$ denotes the orthogonal complement of  $\ker(D\Psi)_{v}$ with respect to $g_{P}$. Given $v \in P$, it follows that 
\begin{equation}
T_{v}P = \ker(\theta)_{v} \oplus j_{v\ast}(\mathfrak{su}(2)),
\end{equation}
see Eq. (\ref{decompositionconnection}). Since the decomposition above is an orthogonal decomposition with respect to $g_{P}$ and $\ker(\theta)_{v} = \ker(D\Psi)_{v}$, it remains to show that 
\begin{equation}
g_{P}(X_{\ast}(v),Y_{\ast}(v)) = <(D\Psi)_{v}(X_{\ast}(v)),(D\Psi)_{v}(Y_{\ast}(v)>,
\end{equation}
where $X_{\ast}$ and $Y_{\ast}$ are the Killing vector fields generated by $X,Y \in \mathfrak{su}(2)$. From this, for all $X, Y \in \mathfrak{su}(2)$ and $\forall v \in P$, denoting $\Psi(v) = g\Gamma_{ADE}$, we obtain
\begin{equation}
\label{submersion}
\begin{split}
<(D\Psi)_{v}(X_{\ast}(v)),(D\Psi)_{v}(Y_{\ast}(v)) > & = <\frac{d}{dt} \Big|_{t = 0}\Psi(v\exp(tX)),\frac{d}{dt} \Big|_{t = 0}\Psi(v\exp(tY)) > \Big |_{\Psi(v)}\\
&= <(Dq)_{g}(X(g)),(Dq)_{g}(Y(g)) > = ( X(g),Y(g))_{g} \\
&= (X, Y)_{{\rm{SU}}(2)}, 
\end{split}
\end{equation}
here we identify $\mathfrak{su}(2)$ with the Lie algebra of right-invariant vector fields on ${\rm{SU}}(2)$, see Eq. (\ref{derivativesub}). On the other hand, by definition of $g_{P}$, we have 
\begin{equation}
\label{submersion1}
g_{P}(X_{\ast}(v),Y_{\ast}(v)) = ( \theta_{v}(X_{\ast}(v)),\theta_{v}(Y_{\ast}(v)) )_{{\rm{SU}}(2)} =  (X, Y)_{{\rm{SU}}(2)}.
\end{equation}
From above, it follows that $\Psi \colon (P,g_{P}) \to ({\rm{SU}}(2) / \Gamma_{ADE}, <\cdot,\cdot>)$ is a proper Riemannian submersion, which concludes the proof.
\end{enumerate}
\end{proof}

\begin{lemma}
\label{representativeKKmetric}
Given $[P,\theta] \in \mathcal{M}_{\text{flat}}(\Sigma,{\rm{SU}}(2))$, for every $(Q,\omega) \in [P,\theta]$, we have an isometry
\begin{equation}
\varphi \colon (Q,g_{Q}) \to (P,g_{P}).
\end{equation}
\end{lemma}
\begin{proof}
By construction, we have 
\begin{equation}
g_{P}:= \pi_{1}^{\ast}(h_{\Sigma}) + (\theta,\theta)_{{\rm{SU}}(2)} \ \ \ \text{and} \ \ \ g_{Q}:= \pi_{2}^{\ast}(h_{\Sigma}) + (\omega,\omega)_{{\rm{SU}}(2)}
\end{equation}
where $\pi_{1} \colon P \to \Sigma$ and $\pi_{2} \colon Q \to \Sigma$ are the natural projections. Since $(Q,\omega) \sim (P,\theta)$, there exists an isomorphism of principal ${\rm{SU}}(2)$-bundles $\varphi \colon P \to Q$, such that $\varphi^{\ast}(\omega) = \theta$. From this, since $\pi_{2} \circ \varphi = \pi_{1}$, it follows that 
\begin{equation}
\begin{split}
\varphi^{\ast}(g_{Q}) & = \varphi^{\ast}\big (  \pi_{2}^{\ast}(h_{\Sigma}) + (\omega,\omega)_{{\rm{SU}}(2)} \big ) \\
& = (\pi_{2} \circ \varphi)^{\ast}(h_{\Sigma}) +  (\varphi^{\ast}(\omega),\varphi^{\ast}(\omega))_{{\rm{SU}}(2)} \\
& = \pi_{1}^{\ast}(h_{\Sigma}) + (\theta,\theta)_{{\rm{SU}}(2)} = g_{P},
\end{split}
\end{equation}
which concludes the proof.
\end{proof}

\section{Gromov-Hausdorff Space}

In this section, we recall the notion of Gromov-Hausdorff distance between metric spaces. The proofs for the results presented in this section can be found in \cite{brin2001course}. 

Given a metric space $(X,d_{X})$, for every pair of subsets $A,B \subset X$, we have the Hausdorff distance between $A$ and $B$ given by
\begin{equation}
d_{H}^{X}(A,B) := \inf \Big \{ \epsilon > 0  \  \big | \ B\subset U_{\epsilon}(A) \ \ \ \text{and} \ \ \ A \subset U_{\epsilon}(B)\Big \}.
\end{equation}
Here $U_{\epsilon}(A) := \{x \in X \ | \  d_{X}(A,x) < \epsilon\}$. It is a classical fact that $d_{H}$ defines a metric on the set of compact subsets of $X$. From above, we have the following definition.

\begin{definition}
The Gromov-Hausdorff distance between two metric spaces $(X,d_{X})$ and $(Y,d_{Y})$ is defined by 
\begin{equation}
d_{GH}((X,d_{X}),(Y,d_{Y})):= \inf_{Z}\inf_{f,g}d_{H}^{Z}(f(X),g(Y)),
\end{equation}
where the infimum is taken over all metric spaces $(Z,d_{Z})$ and all isometric embeddings 
\begin{center}
$f \colon (X,d_{X}) \to (Z,d_{Z})$ \ \ \ and \ \ \ $g \colon (Y,d_{Y}) \to (Z,d_{Z})$.
\end{center}
\end{definition}

In order to compute or estimate $d_{GH}((X,d_{X}),(Y,d_{Y}))$, since the above definition involves cumbersome details even in simple cases, it is more convenient to work with an alternative characterization of the Gromov-Hausdorff distance. To this end, let us introduce the following concept.

\begin{definition}
 Given two metric spaces $(X,d_{X})$ and $(Y,d_{Y})$, a correspondence (or surjective relation) between the underlying sets $X$ and $Y$ is a subset $\mathscr{R} \subseteq  X \times Y$ such that the projections $\pi_{X} \colon X \times Y \to X$ and $\pi_{Y} \colon X \times Y \to Y$ remain surjective when restricted to $\mathscr{R}$.
 \end{definition}
\begin{example}
A straightforward example of correspondence is given by a map $F \colon (X,d_{X}) \to (Y,d_{Y})$ which is surjective. In this setting, we can check that the graph
\begin{equation}
\mathscr{R}_{F} = \{(x,F(x)) \in X \times Y \ | \ x \in X\},
\end{equation}
defines a correspondence between $(X,d_{X})$ and $(Y,d_{Y})$.
\end{example}

\begin{remark}
In what follows, we denote by $\mathcal{R}(X,Y)$ the set of all correspondences between two metric spaces $(X,d_{X})$ and $(Y,d_{Y})$.
\end{remark}

Now we consider the following definition.

\begin{definition}
Let $\mathscr{R} \in \mathcal{R}(X,Y)$ be a correspondence between two metric spaces $(X,d_{X})$ and $(Y,d_{Y})$. The distortion of $\mathscr{R}$ is defined by 
\begin{equation}
{\rm{dis}}(\mathscr{R}) = \sup \Big  \{ |d_{X}(x,x') - d_{Y}(y,y')| \ \ \big | \ \ (x,y), (x',y') \in \mathscr{R} \ \Big\}.
\end{equation}
\end{definition}

\begin{remark}
Notice that, if ${\rm{dis}}(\mathscr{R}) = 0$, then $\mathscr{R}$ is the graph of an isometry.
\end{remark}

From above, we can characterize the Gromov-Hausdorff distance between two metric spaces as follows (e.g. \cite[Theorem 7.3.25]{brin2001course}).

\begin{theorem}
For any two metric spaces $(X,d_{X})$ and $(Y,d_{Y})$, we have 
\begin{equation}
\label{GHdistance}
d_{GH}\big ((X,d_{X}),(Y,d_{Y}) \big ) = \frac{1}{2}\inf \Big  \{ {\rm{dis}}(\mathscr{R}) \ \ \Big | \ \ \mathscr{R}  \in \mathcal{R}(X,Y) \Big\}.
\end{equation}
In other words, $d_{GH}\big ((X,d_{X}),(Y,d_{Y}) \big )$ is equal to the infimum of $r > 0$ for which there exists a correspondence $\mathscr{R}  \in \mathcal{R}(X,Y)$, such that ${\rm{dis}}(\mathscr{R}) < 2r$.
\end{theorem}

The Gromov-Hausdorff distance defines a finite metric on the space of isometry classes of compact metric spaces. In other words, it is nonnegative, symmetric and satisfies the triangle inequality; moreover, $d_{GH}\big ((X,d_{X}),(Y,d_{Y}) \big ) = 0$ if and only if $(X,d_{X})$ and $(Y,d_{Y})$ are isometric. In summary, we have the following theorem (e.g. \cite[Theorem 7.3.30]{brin2001course}). 
\begin{theorem} 
Let $\mathfrak{M}$ be the set of all equivalence classes of isometric compact metric spaces. Then, $(\mathfrak{M},d_{GH})$ is a metric space (Gromov–Hausdorff space).
\end{theorem}

\begin{remark}
We shall denote a point of $\mathfrak{M}$ representing an equivalence class $[(X,d_{X})]$ just by $(X,d_{X})$.
\end{remark}

\section{Superspace of Riemannian Metrics}

Given a smooth manifold $M$ let $\Gamma^{\infty}(S^{2}T^{\ast}M)$ be the real vector space of smooth symmetric $(0,2)$ tensor fields on $M$.

\begin{definition}
The space $\mathcal{R}(M)$ of all complete Riemannian metrics on $M$ is the subspace of $\Gamma^{\infty}(S^{2}T^{\ast}M)$ consisting of all sections which are complete Riemannian metrics on $M$, equipped with the smooth topology of uniform convergence on compact subsets.
\end{definition}

\begin{remark}
\label{Seminorm}
Let $M$ be a compact manifold, fixing some reference Riemannian metric $g_{0}$, we consider the Fréchet topology on $\Gamma^{\infty}(S^{2}T^{\ast}M)$ generated by the family of seminorms $||\cdot||_{C^{k}}, k=0,1,2,\ldots$, defined by 
\begin{equation}
||A||_{C^{0}} = \sup_{x \in M}\sup_{v \neq 0}\frac{|A(x)(v,v)|}{g_{0}(v,v)}, \ \ \text{and} \ \ ||A||_{C^{k}} = ||A||_{C^{0}} + \sum_{j=1}^{k}\sup_{x \in M}|\nabla^jA(x)|_{g_{0}},  \ \ \ \text{if} \ k \geq 1.
\end{equation}
The topology generated by the seminorms above coincides with the smooth topology of uniform convergence on compact subsets. In particular, it follows that the topology generated by the seminorms above does not depend on the choice of $g_{0}$.
\end{remark}
In order to introduce the moduli space of complete Riemannian metrics on $M$, we shall consider the action of the diffeomorphism group ${\rm{Diff}}(M)$ on $\mathcal{R}(M)$ by pulling back metrics, i.e., we consider the action
\begin{equation}
{\rm{Diff}}(M) \times \mathcal{R}(M) \to  \mathcal{R}(M), \ \ \ (\varphi,g) \mapsto \varphi^{\ast}(g).
\end{equation}
From above, we have the following definition.

\begin{definition}
The moduli space $\mathcal{S}(M)$ of complete Riemannian metrics on $M$ is defined by the quotient 
\begin{equation}
\mathcal{S}(M):= \mathcal{R}(M)/{\rm{Diff}}(M),
\end{equation}
equipped with the quotient topology.
\end{definition}

Due to the fact that different Riemannian metrics may have isometry groups of different dimension, the moduli space $\mathcal{S}(M)$ will in general not have any kind of manifold structure. In this work, we are interested in the following realization of $\mathcal{S}(M)$.

\begin{lemma}
\label{continuitymetricdistance}
If $M$ is a compact manifold, then the map
\begin{equation}
\label{injectivemapmetric}
\widetilde{\Phi} \colon \mathcal{S}(M) \to \mathfrak{M}, \ \ \ \ \ [g] \mapsto (M,d_{g}),
\end{equation}
where $d_{g}$ denotes the distance induced by $g$, is continuous and injective. 
\end{lemma}
\begin{proof}
Since $M$ is compact, we have that $\mathcal{R}(M)$ is a convex cone in the Fréchet space $\Gamma^{\infty}(S^{2}T^{\ast}M)$. Let $\Phi \colon \mathcal{R}(M) \to \mathfrak{M}$ be the map defined by
\begin{equation}
\Phi(g) := (M,d_{g}), \ \ \ \forall g \in \mathcal{R}(M).
\end{equation}
Given $g,h \in \mathcal{R}(M)$, considering the identity map ${\rm{id}}_{M} \colon (M,d_{g}) \to (M,d_{h})$, it follows that 
\begin{equation}
\label{C^0-continuity}
d_{GH}((M,d_{g}),(M,d_{h})) \leq \frac{1}{2}{\rm{dis}}(\mathscr{R}_{{\rm{id}}}) = \frac{1}{2}\sup_{x,y \in M}|d_{g}(x,y) - d_{h}(x,y)| = \frac{1}{2}||d_{g} - d_{h}||_{C^{0}}.
\end{equation}
Notice that $d_{g},d_{h}\in (C^{0}(M \times M),||-||_{C^{0}})$. Let us verify that the map $g \mapsto d_{g}$ is continuous with respect to the $C^{0}$-topology. Given $g,h \in \mathcal{R}(M)$, taking $g$ as the reference metric for $||\cdot||_{C^{0}}$, see Remark \ref{Seminorm}, it follows that 
\begin{equation}
- ||h-g||_{C^{0}} \leq \frac{(h-g)(v,v)}{g(v,v)} \leq ||h-g||_{C^{0}}, \ \ v \neq 0\in TM.
\end{equation}
From above, it follows that 
\begin{equation}
(1-||h-g||_{C^{0}})g(v,v) \leq h(v,v) \leq (1+||h-g||_{C^{0}})g(v,v), \ \ v \neq 0\in TM.
\end{equation}
Considering $||h-g||_{C^{0}}$ sufficiently small, i.e., considering $h$ sufficiently close to $g$, we have 
\begin{equation}
(\sqrt{1-||h-g||_{C^{0}}})d_{g}(x,y)\leq d_{h}(x,y) \leq (\sqrt{1+||h-g||_{C^{0}}})d_{g}(x,y).
\end{equation}
Thus, it follows that 
\begin{enumerate}
\item[(a)] $d_{h}(x,y) - d_{g}(x,y) \leq \big [ \sqrt{1 + ||h-g||_{C^{0}}} - 1\big ] {\rm{diam}}(M,g)$,

\item[(b)] $d_{g}(x,y) - d_{h}(x,y) \leq \big [1- \sqrt{1 - ||h-g||_{C^{0}}}\big ] {\rm{diam}}(M,g)$.
\end{enumerate}
From (a) and (b), we obtain
\begin{equation}
||d_{h} - d_{g}||_{C^{0}} \leq \max\Big\{  \sqrt{1 + ||h-g||_{C^{0}}} - 1, 1- \sqrt{1 - ||h-g||_{C^{0}}}\Big\}{\rm{diam}}(M,g).
\end{equation}
Therefore, if $0 \leq  ||h-g||_{C^{0}} \leq \frac{1}{2}$, it follows that 
\begin{equation}
||d_{h} - d_{g}||_{C^{0}} \leq {\rm{diam}}(M,g)||h-g||_{C^{0}}.
\end{equation}
Hence, given $g \in \mathcal{R}(M)$, $\forall \epsilon > 0$, we can take
\begin{equation}
\delta:= \min \Big \{\frac{1}{2}, \frac{\epsilon}{{\rm{diam}}(M,g)}\Big\},
\end{equation}
so that 
\begin{equation}
\forall h \in \mathcal{R}(M), \ \text{if} \ ||h-g||_{C^{0}} < \delta, \ \text{then} \ ||d_{h} - d_{g}||_{C^{0}} < \epsilon,
\end{equation}
in other words, the map $g \mapsto d_{g}$ is continuous with respect to the $C^{0}$-topology. This last fact combined with Eq. (\ref{C^0-continuity}) implies that $\Phi$ is continuous in the $C^{0}$-topology. Observing that the Fréchet topology is finer (stronger) than the $C^{0}$-topology, we conclude that $\Phi$ is continuous in the Fréchet topology on $\mathcal{R}(M)$. Further, since $\Phi$ is constant on the fibers of the quotient map $\varpi \colon \mathcal{R}(M) \to \mathcal{S}(M)$, it follows that the map $\widetilde{\Phi} \colon \mathcal{S}(M) \to \mathfrak{M}$, defined in Eq. (\ref{injectivemapmetric}), satisfies
\begin{equation}
\widetilde{\Phi} \circ \varpi = \Phi,
\end{equation}
i.e., $\widetilde{\Phi}: [g] \mapsto (M,d_{g})$ is continuous. The injectivity of $\widetilde{\Phi}$ follows from the Myers–Steenrod theorem, see for instance \cite[Theorem 1]{myers1939group}.
\end{proof}

\begin{remark}
In the setting of the above result, we shall identify $\mathcal{S}(M)$ with ${\rm{Im}}(\widetilde{\Phi}) \subset \mathfrak{M}$ as a set, in other words, we consider
\begin{equation}
\mathcal{S}(M) = \big \{ (M,d_{g}) \  | \ g \in \mathcal{R}(M)\big \}.
\end{equation}
Our main result concerns the study of the Gromov-Hausdorff closure of $\mathcal{S}(P)$ in the case that $(P,\theta)$ is a flat ${\rm{SU}}(2)$-bundle over a closed surface $\Sigma$ of genus $g(\Sigma) \geq 2$. The continuity of the injective map $\widetilde{\Phi} \colon \mathcal{S}(P) \to \mathfrak{M}$ provided by Lemma \ref{continuitymetricdistance} allows us to study convergence issues of sequences of Riemannian metrics using the Gromov-Hausdorff topology.
\end{remark}

\section{Proof of Main Result}

In this section, we present the proof of the main result of this work.

\begin{proof}[Proof of Theorem \ref{mainresult}] Given a homogeneous spherical $3$-manifold $(S,g_{S})$, we have an isometry 
\begin{equation}
\varphi \colon (S,g_{S}) \to ({\rm{SU}}(2)/\pi_{1}(S),<\cdot,\cdot>),
\end{equation}
where $\pi_{1}(S) = \Gamma_{ADE}\subset {\rm{SU}}(2)$ is a binary polyhedral subgroup provided by the ADE classification. As we have seen in Example \ref{representationfiniteADE}, there exists a representation $\varrho \colon \pi_{1}(\Sigma) \to {\rm{SU}}(2)$, for some closed hyperbolic surface $\Sigma \in \mathcal{M}_{g}$ ($g \geq 2$), such that ${\rm{Im}}(\varrho) = \pi_{1}(S)$. From this, there exists $[P,\theta] \in \mathcal{M}_{\text{flat}}(\Sigma,{\rm{SU}}(2))$, such that 
\begin{equation}
\pi_{1}(S) \cong {\rm{Hol}}_{u}(\theta), \ \ \ \forall u \in P.
\end{equation}
Fixing the Poincaré Riemannian metric $h_{\Sigma}$ on $\Sigma$, we define the following sequence of Riemannian metrics
\begin{equation}
g_{P,n}:= \frac{1}{n}\pi^{\ast}(h_{\Sigma}) + (\theta,\theta)_{{\rm{SU}}(2)}, \ \ \ n =1,2,\ldots,
\end{equation}
on the manifold $P$. From Lemma \ref{submersiondistortion}, we have a proper Riemannian submersion
\begin{equation}
\Psi \colon (P,g_{P,n}) \to ({\rm{SU}}(2)/\pi_{1}(S),<\cdot,\cdot>), \ \ n =1,2,\ldots.
\end{equation}
It is worth pointing out that, since $\ker(\Psi_{\ast}) = \ker(\theta) \cong T\Sigma$, the map $\Psi$ does not depend on $n =1,2,\ldots$. Moreover, $(P,d_{g_{P,n}}) \in \mathfrak{M}$ does not depend on the choice of the representative in the equivalence class $[P,\theta] \in \mathcal{M}_{\text{flat}}(\Sigma,{\rm{SU}}(2))$, see Lemma \ref{representativeKKmetric}.

Given $a,b \in P$, since $ ({\rm{SU}}(2)/\pi_{1}(S),<\cdot,\cdot>)$ is a complete Riemannian manifold, there exists a minimizing geodesic $\gamma \colon [0,1] \to {\rm{SU}}(2)/\pi_{1}(S)$, such that 
\begin{equation}
d_{<\cdot,\cdot> }(\Psi(a),\Psi(b)) = \int_{0}^{1}\sqrt{<\dot{\gamma},\dot{\gamma}>}{\rm{d}}s.
\end{equation}
On the other hand, considering, respectively, the vertical and horizontal distributions on $P$ given by
\begin{equation}
\mathcal{V} := \ker(D\Psi) \ \ \ \ \text{and} \ \ \ \ \mathcal{H} := \ker(D\Psi)^{\perp_{g_{P,n}}}, 
\end{equation}
since $\mathcal{H}$ is Ehresmann-complete, there exists an $\mathcal{H}$-horizontal lift $\beta$ of $\gamma$, such that $\beta(0) = a$, see for instance \cite[\S 2.1]{pastore2004riemannian}, \cite{besse2007einstein}. Let us denote $z = \beta(1) \in P_{b}(\theta)$. From this, we have
\begin{equation}
d_{g_{P,n}}(a,z) \leq \int_{0}^{1}\sqrt{g_{P,n}(\dot{\beta},\dot{\beta})}{\rm{d}}s = \int_{0}^{1}\sqrt{<\dot{\gamma},\dot{\gamma}>}{\rm{d}}s = d_{<\cdot,\cdot> }(\Psi(a),\Psi(b)),
\end{equation}
so we conclude that 
\begin{equation}
d_{g_{P,n}}(a,b) \leq d_{g_{P,n}}(a,z) + d_{g_{P,n}}(z,b) \leq d_{<\cdot,\cdot> }(\Psi(a),\Psi(b)) +   d_{g_{P,n}}(z,b).
\end{equation}
In other words, we obtain
\begin{equation}
d_{g_{P,n}}(a,b) - d_{<\cdot,\cdot> }(\Psi(a),\Psi(b)) \leq d_{g_{P,n}}(z,b).
\end{equation}
Since $z,b \in P_{b}(\theta)$, we have a $\theta$-horizontal curve $\alpha \colon [0,1] \to P_{b}(\theta)$, such that $\alpha(0)=z$ and $\alpha(1) = b$. 

\begin{center}
\begin{figure}[H]
\centering\includegraphics[scale = .10]{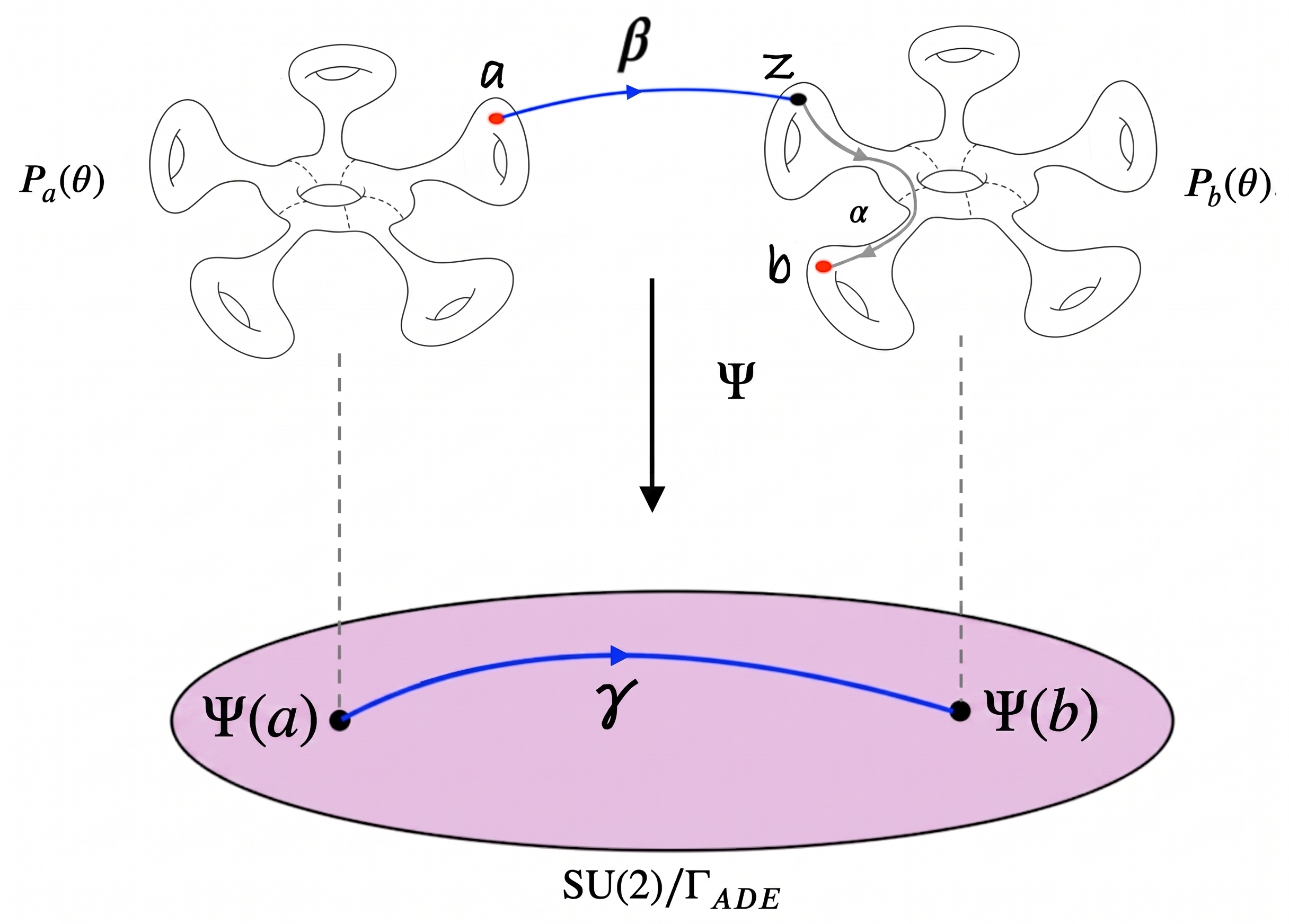}
\caption{The concatenation of $\beta$ with $\alpha$ provides a path in $P$ joining $a$ and $b$. }
\label{Levantamento}
\end{figure}
\end{center}

Thus, since $\dot{\alpha}(s) \in \ker(\theta)_{\alpha(s)}$, $\forall s \in [0,1]$, we obtain
\begin{equation}
d_{g_{P,n}}(a,b) - d_{<\cdot,\cdot> }(\Psi(a),\Psi(b)) \leq d_{g_{P,n}}(z,b) \leq \frac{1}{\sqrt{n}}\int_{0}^{1}\sqrt{h_{\Sigma}(\pi_{\ast}(\dot{\alpha}),\pi_{\ast}(\dot{\alpha}))}{\rm{d}}s.
\end{equation}
From the fact that $\Psi$ is a Riemannian submersion, it follows that $d_{<\cdot,\cdot> }(\Psi(a),\Psi(b)) \leq d_{g_{P,n}}(a,b)$. Therefore, we obtain the following inequality 
\begin{equation}
\big | d_{g_{P,n}}(a,b) - d_{<\cdot,\cdot> }(\Psi(a),\Psi(b))\big | \leq \frac{1}{\sqrt{n}}\int_{0}^{1}\sqrt{h_{\Sigma}(\pi_{\ast}(\dot{\alpha}),\pi_{\ast}(\dot{\alpha}))}{\rm{d}}s.
\end{equation}
Now observing that $z,b \in P_{b}(\theta)$ and that $\pi^{\ast}(h_{\Sigma})$ defines a Riemannian metric on $P_{b}(\theta)$, taking the infimum on the right-hand side above over all curves $\alpha \colon I \to P_{b}(\theta)$ connecting $z$ and $b$, we conclude that 
\begin{equation}
\big | d_{g_{P,n}}(a,b) - d_{<\cdot,\cdot> }(\Psi(a),\Psi(b))\big | \leq \frac{1}{\sqrt{n}} d_{\pi^{\ast}(h_{\Sigma})}(z,b) \leq \frac{{\rm{diam}}(P_{b}(\theta),\pi_{b}^{\ast}(h_{\Sigma}))}{\sqrt{n}}.
\end{equation}
Here, we denote $\pi_{b} = \pi|_{P_{b}(\theta)}$. Notice that, since $\theta$ is flat, it follows that $T P_{b}(\theta) = \ker(\theta)|_{P_{b}(\theta)}$. Hence, every smooth curve $\alpha \colon I \to P_{b}(\theta)$ is $\theta$-horizontal. Since $(P_{b}(\theta),\pi_{b}^{\ast}(h_{\Sigma}))$ is a closed hyperbolic surface, from Theorem \ref{diametersystole}, we have
\begin{equation}
{\rm{diam}}(P_{b}(\theta),\pi_{b}^{\ast}(h_{\Sigma})) \leq \frac{{\rm{Area}}(P_{b}(\theta))}{{\rm{sys}}_{1}(P_{b}(\theta),\pi_{b}^{\ast}(h_{\Sigma}))} \leq \frac{|\pi_{1}(S)|{\rm{Area}}(\Sigma)}{{\rm{sys}}_{1}(\Sigma,h_{\Sigma})},
\end{equation}
see Remark \ref{diamsystoleineq}. From above, it follows that 
\begin{equation}
\big | d_{g_{P,n}}(a,b) - d_{<\cdot,\cdot> }(\Psi(a),\Psi(b))\big | \leq \frac{1}{\sqrt{n}} \frac{|\pi_{1}(S)|{\rm{Area}}(\Sigma)}{{\rm{sys}}_{1}(\Sigma,h_{\Sigma})}.
\end{equation}
Therefore, considering the correspondence 
\begin{equation}
\mathscr{R}_{\Psi}:= \Big \{ (a,\Psi(a)) \in P \times ({\rm{SU}}(2)/\pi_{1}(S)) \ \Big |  \ a \in P\Big \},
\end{equation}
we conclude that 
\begin{equation}
d_{{\rm{GH}}}\big ( (S,d_{g_{S}}),(P,d_{g_{P,n}}) \big) \leq \frac{1}{2} {\rm{dis}}(\mathscr{R}_{\Psi}) \leq \frac{1}{2\sqrt{n}} \frac{|\pi_{1}(S)|{\rm{Area}}(\Sigma)}{{\rm{sys}}_{1}(\Sigma,h_{\Sigma})}.
\end{equation}
Hence, it follows that
\begin{equation}
d_{{\rm{GH}}}\big ( (S,d_{g_{S}}),(P,d_{g_{P,n}}) \big) \to 0 \ \ \ \text{as} \ \ \ n \to + \infty,
\end{equation}
that is, $(S,d_{g_{S}}) \in \overline{\mathcal{S}(P)}^{GH} \subset \mathfrak{M}$. In particular, since $g_{P} = g_{P,1}$, we have
\begin{equation}
d_{{\rm{GH}}}\big ( (S,d_{g_{S}}),(P,d_{g_{P}}) \big ) \leq \frac{1}{2} \frac{|\pi_{1}(S)|{\rm{Area}}(\Sigma)}{{\rm{sys}}_{1}(\Sigma,h_{\Sigma})},
\end{equation}
which concludes the proof.
\end{proof}

\section{From Bolza Surface to Poincaré Homology Sphere}

In this section, we present a simple illustration of the mechanism provided by Theorem \ref{mainresult}. In order to do so, we choose two distinguished examples of negatively and positively curved manifolds, namely, the Bolza surface and the Poincaré homology sphere, respectively.

The Bolza surface $\Sigma$ is a Riemann surface of genus $2$ with a holomorphic automorphism group of order 48, the highest for this genus \cite{bolza1887binary}, \cite{jenni1984ersten}, \cite{schmutz1993reimann}.
The surface $\Sigma$ can be viewed as the smooth completion of its affine form
\begin{equation}
\label{affineBolza}
y^{2} = x^{5} - x,
\end{equation}
where $(x,y) \in \mathbbm{C}^{2}$. Here $\Sigma$ is a double cover of the Riemann sphere ramified over the vertices of the regular inscribed octahedron (Fig. \ref{Ram_Points}).
\begin{center}
\begin{figure}[H]
\centering\includegraphics[scale = .12]{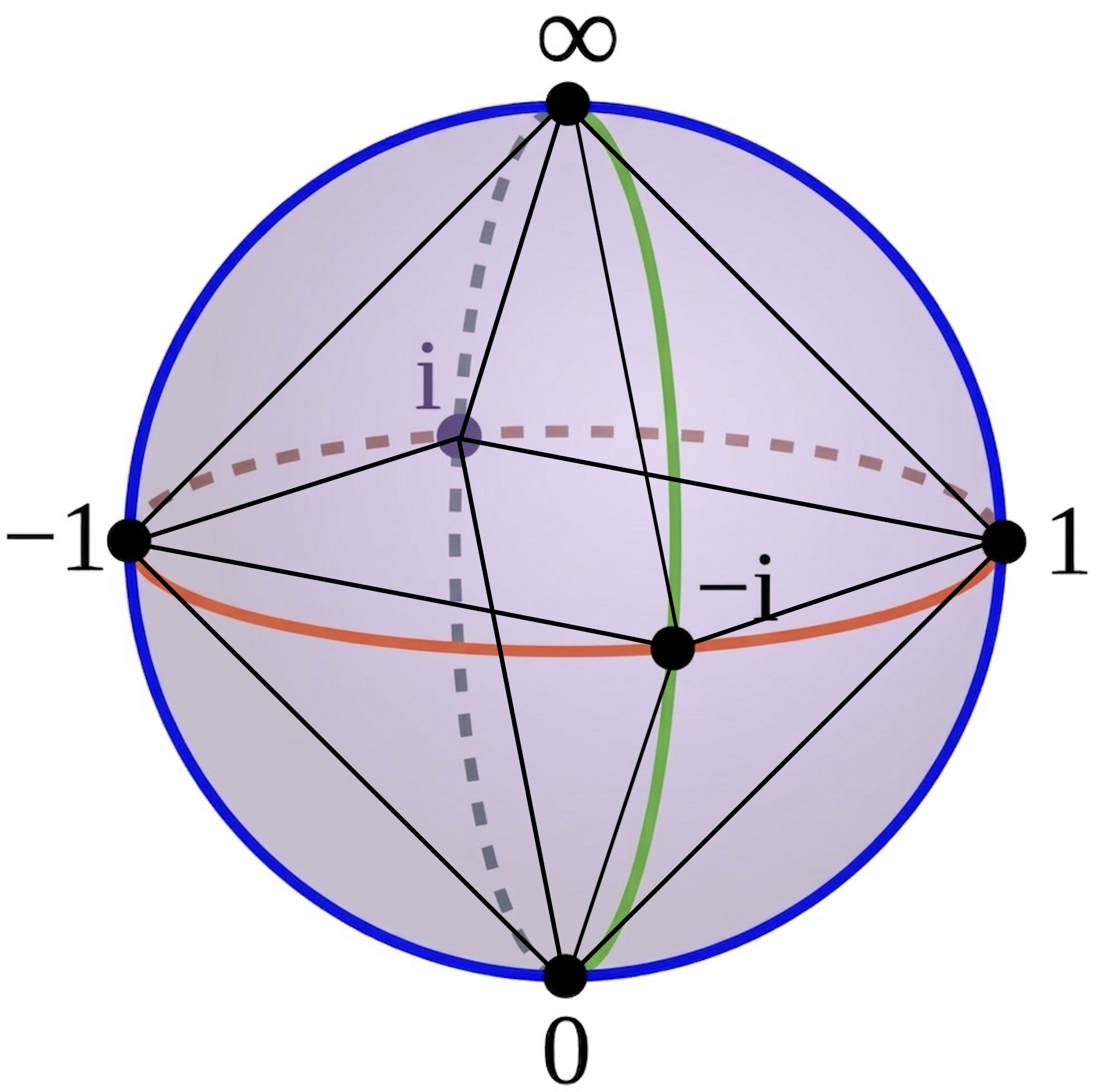}
\caption{It is immediate from the presentation in Eq. (\ref{affineBolza}) that the branch points are $0, \pm 1, \pm i,\infty$.}
\label{Ram_Points}
\end{figure}
\end{center}

On the other hand, the Poincaré homology sphere $S$ can be described as a quotient of ${\rm{SU}}(2)$ by the binary icosahedral group
\begin{equation}
I^{\ast} = \langle A, B \mid A^{3} = B^{5} = (AB)^2\rangle,
\end{equation}
namely, $S = {\rm{SU}}(2)/I^{\ast}$, e.g. \cite{kirby1979eight}. From this, considering the representation  $\varrho \colon \pi_{1}(\Sigma) \to {\rm{SU}}(2)$ by setting
\begin{equation}
\varrho(a_{1}) = A, \ \varrho(b_{1}) = B, \ \varrho(a_{2}) = B, \ \varrho(b_{2}) = A.
\end{equation}
Notice that $\pi_{1}(\Sigma) = \langle a_{1},b_{1},a_{2},b_{2} | [a_{1},b_{1}][a_{2},b_{2}] = 1\rangle$. Thus, we obtain a flat principal ${\rm{SU}}(2)$-bundle $(P_{\varrho},\theta_{\varrho})$ over $\Sigma$, such that 
\begin{equation}
P_{\varrho} = \mathbbm{H}^{2} \times_{\varrho} {\rm{SU}}(2).
\end{equation}
Taking the sequence of metrics 
\begin{equation}
g_{n} = \frac{1}{n}\pi^{\ast}(h_{\Sigma}) + (\theta_{\varrho},\theta_{\varrho})_{{\rm{SU}}(2)}, \ \ n = 1,2,\ldots
\end{equation}
on $P_{\varrho}$, and observing that 
\begin{equation}
|\pi_{1}(S)| = |I^{\ast}| = 120, \ \ \ {\rm{Area}}(\Sigma) = 4\pi, \ \ \ {\rm{sys}}_{1}(\Sigma) = 2\arccosh(1 + \sqrt{2}),
\end{equation}
cf. \cite[Theorem 5.2]{schmutz1993reimann}, we obtain the following upper bound
\begin{equation}
d_{{\rm{GH}}}\big ( (S,d_{g_{S}}),(P_{\varrho},d_{g_{n}})  \big) \leq \frac{1}{2\sqrt{n}}\frac{|\pi_{1}(S)|{\rm{Area}}(\Sigma)}{{\rm{sys}}_{1}(\Sigma,h_{\Sigma})} = \frac{1}{\sqrt{n}} \frac{120 \pi}{\arccosh(1 + \sqrt{2})}.
\end{equation}
Since the Bolza surface is the global maximum of ${\rm{sys}}_{1} \colon \mathcal{M}_{2} \to \mathbbm{R}_{>0}$, the upper bound error obtained from Theorem \ref{mainresult} described above is optimal.

\begin{remark}
In the case that $g = 3$, since ${\rm{sys}}_{1} \colon \mathcal{M}_{3} \to \mathbbm{R}_{>0}$ attains a local maximum at the Klein quartic
\begin{equation}
K: x^{3}y + y^{3}z + z^{3}x = 0, \ \ \ \ [x:y:z] \in \mathbbm{P}^{2},
\end{equation}
see for instance \cite{schmutz1993reimann}, we obtain from Theorem \ref{mainresult}, at least close to $K \in\mathcal{M}_{3}$, an optimal upper bound error by a similar computation as in the case of the Bolza surface. 
\end{remark}

\bibliographystyle{alpha}
\bibliography{reference.bib}

\end{document}